\documentclass[a4paper, 11pt, oneside, sffamily]{article}

\usepackage{graphicx}
\usepackage{wrapfig}

\usepackage[a4paper, left=3.5cm, right=3.5cm, top=4cm, bottom=5cm]{geometry}
\makeatletter
\renewcommand{\section}{\@startsection {section}{1}{\z@}%
             {-3.5ex \@plus -1ex \@minus -.2ex}%
             {2.3ex \@plus.2ex}%
             {\vspace{2em}\center\normalfont\normalsize\bfseries}}
\renewcommand{\subsection}{\@startsection
{subsection}{1}{\z@}%
             {-3.5ex \@plus -1ex \@minus -.2ex}%
             {2.3ex \@plus.2ex}%
             {\normalfont\normalsize\bfseries}}
\makeatother
\usepackage{tensor}
\usepackage{amsmath,latexsym}
\allowdisplaybreaks
\usepackage{amssymb}
\usepackage{physics}
\usepackage{dsfont} 
\usepackage{tikz}
	\usetikzlibrary{cd}
\usepackage{tikzit}

\tikzstyle{box}=[fill=white, draw=black, shape=rectangle]

\tikzstyle{red}=[-, draw=white, double={rgb,255: red,191; green,0; blue,64}]
\tikzstyle{thick line}=[-, thick]

\usepackage{ytableau} 
	\ytableausetup{centertableaux}

\usepackage[T1]{fontenc}
\usepackage[french, english]{babel}
\usepackage{lipsum}
\usepackage{xcolor}
	\definecolor{text-black}{rgb}{0.1,0.1,0.1}
	\color{text-black}
\usepackage{makecell} 
\setcellgapes{5pt}

\usepackage{amsthm}
	\theoremstyle{definition}
	\newtheorem{definition}{Definition}[section]
	\newtheorem{theorem}[definition]{Theorem}
	\newtheorem{remark}[definition]{Remark}
	\newtheorem{lemma}[definition]{Lemma}
	\newtheorem{proposition}[definition]{Proposition}
	\newtheorem{corollary}[definition]{Corollary}
	\newtheorem{example}[definition]{Example}
	\newtheorem*{example*}{Example}
	
\usepackage[utf8]{inputenc}

\definecolor{colorlink}{RGB}{150, 10, 10}
\usepackage[colorlinks,bookmarks=false,linkcolor=colorlink,urlcolor=colorlink, citecolor=colorlink]{hyperref}

\usepackage{units}
\usepackage[squaren, cdot]{SIunits}

\usepackage{verbatim}
\usepackage{verbdef} 

\begin{document}

\renewcommand{\refname}{}

\newcommand{\catname}[1]{\normalfont\mathbf{#1}}
\newcommand{\calcatname}[1]{\mathcal{#1}}

\definecolor{blackred}{RGB}{150, 10, 10}
\newcommand{\todo}[1]{
	\vspace{\baselineskip}
		\hspace{0.05\textwidth}\begin{minipage}{0.9\textwidth}\hspace{-1em}\color{blackred}{\textit{To do}: #1}\end{minipage}
	\vspace{\baselineskip}
	}

\definecolor{blackblue}{RGB}{10, 10, 150}
\newcommand{\question}[1]{
	\vspace{\baselineskip}
		\hspace{0.05\textwidth}\begin{minipage}{0.9\textwidth}\hspace{-1em}\color{blackblue}{\textit{Question}: #1}\end{minipage}
	\vspace{\baselineskip}
	}

\begin{titlepage}

\title{\huge\textbf{\sffamily Constructing new open-closed TQFTs from the interpolation of symmetric monoidal categories}}
\author{
    \vspace{1em}
    {\sffamily \LARGE Barthélémy Neyra} \\ 
    {\large \href{mailto:b.neyra.21@abdn.ac.uk}{b.neyra.21@abdn.ac.uk}} \\
    {\small \textit{Institute of Mathematics, University of Aberdeen,}} \\
    {\small \textit{Fraser Noble Building, Aberdeen AB24 3UE, UK}}
}
\date{}

\maketitle
\thispagestyle{empty}

\begin{center}\textbf{Abstract}\end{center}

For any symmetric monoidal category $\calcatname{D}$, Lauda and Pfeiffer showed the equivalence between the $\calcatname{D}$-valued open-closed 2-dimensional TQFTs and the so-called knowledgeable Frobenius algebras (KFAs) in $\calcatname{D}$. 
Each KFA in $\calcatname{D}=\catname{Vec}_{\mathds{K}}$ provides a sequence of scalars indexed by the set $\mathds{N}^2$ of diffeomorphism classes of connected endocobordisms of the empty set, given by evaluation by the associated TQFT on each such cobordism class. From an arbitrary sequence $\chi=(\chi_{g,w})_{g,w\in\mathds{N}}$, we build a symmetric monoidal category $\calcatname{C}_{\chi}$ -- with unit object $\textbf{1}$ satisfying $\text{End}_{\calcatname{C}_{\chi}}(\textbf{1})\cong \mathds{K}$ -- generated by a KFA object affording this sequence. We then determine which sequences $\chi$ produce semisimple abelian categories $\calcatname{C}_{\chi}$ with finite-dimensional hom-spaces. These form a family of categories interpolating the categories of representations of automorphism groups of certain KFAs in $\catname{Vec}_{\mathds{K}}$.

\let\endtitlepage\relax

\end{titlepage}

\begingroup
\let\cleardoublepage\relax
\let\clearpage\relax
\tableofcontents
\endgroup

\begin{section}{Introduction}
\label{section-introduction}

It is well documented that $2$-dimensional quantum field theories (TQFTs) are a powerful machinery to produce algebraic structures. In general, it takes the following form of statement: "For any symmetric monoidal category $\calcatname{D}$, the $\calcatname{D}$-valued TQFTs from the bordism category $\calcatname{A}$ are equivalent to $\mathcal{T}$ objects in $\calcatname{D}$", where $\mathcal{T}$ is to be replaced by a particular name of algebraic structure, or, more exactly, by a \textit{theory}. Examples are given in table \ref{table:1}.

\begin{table}[h!]
\centering
\makegapedcells
\begin{tabular}{ c | c | c }
Bordism category $\calcatname{A}$ & Theory $\mathcal{T}$ & Ref. \\
\hline
\hline
oriented closed & commutative Frobenius algebra & \cite{kok} \\
\hline
$r$-spin closed & $\Lambda_r$-Frobenius algebra & \cite{stern-szegedy-2022} \\
\hline
oriented open-closed & knowledgeable Frobenius algebra & \cite{lauda-pfeiffer-1} \\
\end{tabular}

\caption{Some bordism categories and their corresponding theories.}
\label{table:1}
\end{table}

For instance, the $\calcatname{D}$-valued oriented closed TQFTs are equivalent to commutative Frobenius algebra objects in $\calcatname{D}$. We say that a commutative Frobenius algebra object in $\calcatname{D}$ is a \textit{model} in $\calcatname{D}$ for the theory of commutative Frobenius algebras. \\

In \cite{meir-interpolation}, Meir showed how to build linear, rigid, symmetric monoidal categories from an arbitrary theory $\mathcal{T}$. The construction involves a ring denoted $\mathds{K}[X]_{\text{aug}}^{\mathcal{T}}$. Each character $\chi:\mathds{K}[X]_{\text{aug}}^{\mathcal{T}}\rightarrow \mathds{K}$ (equivalently, each $\mathds{K}$-valued sequence, also denoted $\chi$, indexed by a set of generators of $\mathds{K}[X]_{\text{aug}}^{\mathcal{T}}$) can be used to produce a symmetric monoidal category $\calcatname{C}_{\chi}$. Amongst other properties, $\calcatname{C}_{\chi}$ is $\mathds{K}$-linear and has its unit object~$\textbf{1}$ satisfying $\text{End}_{\calcatname{C}_{\chi}}(\textbf{1})\cong \mathds{K}$. It is generated by a model for $\mathcal{T}$, from which one can retrieve the character $\chi$ explicitly. Conversely, any model $V$ for $\mathcal{T}$ in a tensor category $\calcatname{D}$ (cf. definition \ref{definition-tensor-category}) affords a character $\chi_V:\mathds{K}[X]_{\text{aug}}^{\mathcal{T}}\rightarrow \mathds{K}$. Some characters/sequences $\chi$ are \textit{good}, in the sense that their associated categories $\calcatname{C}_{\chi}$ are semisimple abelian with finite-dimensional hom-spaces, or equivalently (cf. definition \ref{theorem-udi-good-character}), if they are afforded by models of $\mathcal{T}$ in an abelian category with finite-dimensional hom-spaces. \\

Different authors have used this construction in particular cases. The categories $\calcatname{C}_{\chi}$ are intimately related to categories of representations of groups and give a context in which the notion of interpolation of categories of representations is meaningful. In fact, it is shown in \cite{meir-interpolation} that the interpolating families of categories of representations $\catname{Rep}(S_t)$, $\catname{Rep}(O_t)$, and $\catname{Rep}(GL_t)$, described by Deligne in \cite{deligne-category-gl} are part of this framework. In the context of Lie algebras, Vogel \cite{vogel-universal} and his projective plane parameterising (in a sense that is yet to be made precise) simple quadratic Lie algebras fit also in this picture. \\

In 2020, Khovanov, Ostrik, and Kononov \cite{kok} studied the sequences for the theory of commutative Frobenius algebras (the first line in table \ref{table:1}), giving a criterion discriminating the good sequences from the others. This is reviewed in \cite{meir-interpolation}, section 8.6. In \cite{meir-invariant-r-spin} the second line of table \ref{table:1} is studied, namely the $\Lambda_r$-Frobenius algebras and their corresponding good sequences. \\

In this paper, we continue the current thread of classifying the good sequences, this time in the context of oriented open-closed TQFTs, i.e., the third line in table \ref{table:1}. We will see that scalar invariants for $\catname{Vec}_{\mathds{K}}$-valued open-closed TQFTs cannot be arbitrary. However, it is possible to build other target categories allowing more general -- unconstrained concerning invariants -- open-closed TQFTs. To achieve this aim, we review in section \ref{section-meir-tool}, the construction of the category $\calcatname{C}_{\chi}$ given a theory $\mathcal{T}$ and a character $\chi:\mathds{K}[X]_{\text{aug}}^{\mathcal{T}}\rightarrow\mathds{K}$. We give a criterion for $\chi$ to be good. \\

In section \ref{section-open-closed}, we describe the bordism category of oriented open-closed cobordisms as the symmetric monoidal category generated by a knowledgeable Frobenius algebra (KFA) object. The theory $\mathcal{T}_{\text{KFA}}$ of KFAs has corresponding $\mathds{K}[X]_{\text{aug}}^{\mathcal{T}_{\text{KFA}}}$ the polynomial ring on the set $\mathds{N}^2$, which is precisely the polynomial ring on the set of connected endocobordisms of the empty set. 
We give some examples of KFA objects in different categories (in particular in the category of finite-dimensional $\mathds{K}$-vector spaces $\catname{Vec}_{\mathds{K}}$). \\

Any $\catname{Vec}_{\mathds{K}}$-valued open-closed TQFT gives thus a character $\chi:\mathds{K}[\mathds{N}^2]\rightarrow \mathds{K}$. It is also true that any character $\chi$ extends to a $\calcatname{C}_{\chi}$-valued open-closed TQFT, even when $\chi$ cannot be afforded by any KFA in $\catname{Vec}_{\mathds{K}}$. Meir's framework allows the construction of TQFTs associating arbitrary scalar invariants to the endocobordisms of the empty set. \\

In section \ref{section-good-knfrob}, we prove the main result of this paper, which is a criterion discriminating the good sequences.
\begin{theorem} \label{theorem-main}
    A $\mathds{K}$-valued sequence $\chi = (\chi_{g,w})_{g,w\in\mathds{N}}$ is good if and only if the associated power series 
    \[ f_{\chi}(X,Y) := \sum_{g,w\geqslant 0} \chi_{g,w} X^g Y^w \]
    belongs to the $\mathds{K}$-vector subspace $\mathcal{G}$ of $\mathds{K}[[X,Y]]$ spanned by the following set: 
    \[
         \left\{1, X, Y, Y^2 \right\} \cup \left\{ \frac{1}{1-\lambda X}\frac{1}{1-\mu Y} : \lambda, \mu \in \mathds{K}, \lambda \neq 0 \right\}.
    \]
    We call $f_{\chi}(X,Y)$ the \textit{generating function} for $\chi$.
\end{theorem}

In section \ref{section-rep-categories}, we study the categories $\calcatname{C}_{\chi}$ when $\chi$ is good, and in particular we link them to different categories of representations. This is recapitulated in table \ref{table:recap} below.

\begin{table}[!h]
\centering
\makegapedcells
    \begin{tabular}{ c | c }
        Generating function $f_{\chi}(X,Y)$ & $\calcatname{C}_{\chi}$ is equivalent to... \\
        \hline
        \hline
        $\alpha_1$, with $\alpha_1\neq 0$ & \makecell{the subcategory $\calcatname{C}$ of $\catname{Rep}(\mathfrak{osp}(1|2))$ \\ generated by its defining representation; \\ cf. section 5.5 in \cite{kok}.} \\
        \hline
        \makecell{$\alpha_1+\alpha_X X$, \\ $\alpha_X \neq 0$}  &  $\catname{Rep}(O_{\alpha_X-2})$; cf. section 7.1 in \cite{kok}. \\
        \hline
        \makecell{$\alpha_1 + \alpha_X X + \alpha_Y Y$, \\ $\alpha_Y\neq0$} & $\calcatname{C} \boxtimes \catname{Rep}(O_{\alpha_X-2})$. \\
        \hline
        \makecell{$\alpha_1 + \alpha_X X + \alpha_Y Y + \alpha_{Y^2} Y^2$, \\ $\alpha_{Y^2}\neq0$} & $\catname{Rep}(O_{\alpha_{Y^2}-2}) \boxtimes \catname{Rep}(O_{\alpha_X-3})$. \\
        \hline
        \makecell{$\frac{t}{\lambda}\frac{1}{1-\lambda X}$, \\ $\lambda,t\neq0$} & $\catname{Rep}(S_{t})$; cf. section 6.1 in \cite{kok}. \\
        \hline
        \makecell{$\frac{1}{a^2}\frac{1}{1-a^2 X}\frac{1}{1-\mu Y}$, \\ $a\neq0$} & $\catname{Rep}(PGL_{\mu/a})$. \\
        \hline
        \makecell{$\frac{t}{\lambda}\frac{1}{1-\lambda X}\frac{1}{1-\mu Y}$, \\ $\lambda,\mu\neq0$, $t\neq0,1$} & not known. \\
    \end{tabular}

    \caption{Summary of the equivalences between the category $\calcatname{C}_{\chi}$ and other categories of representations when $\chi$ is good.}
    \label{table:recap}
\end{table}

\end{section}

\begin{section}{Preliminaries}

We fix an algebraically closed field $\mathds{K}$ of characteristic 0. We use the following definitions. 

\begin{definition} \label {definition-tensor-category} (Tensor category)
    A \textit{tensor category} is a category $\calcatname{C}$ that is symmetric monoidal (with unit object $\textbf{1}$), rigid (each object has a dual), $\mathds{K}$-linear, and with endomorphism space $\text{End}_{\calcatname{C}}(\textbf{1})\cong \mathds{K}$. 
\end{definition}
In a rigid symmetric monoidal category $\mathcal{C}$, we have in particular the notion of trace $\mathrm{Tr}$ of endomorphisms. For $A\in \text{Ob}(\mathcal{C})$ and $T\in \mathcal{C}(A,A)$, $\mathrm{Tr}(T)$ is defined by the composition
\[ \textbf{1}\xrightarrow{coev} A\otimes A^{*} \xrightarrow{T\otimes \mathds{1}_{A^{*}}} A\otimes A^{*} \xrightarrow\beta A^{*}\otimes A \xrightarrow{ev} \textbf{1} \] 
which is a morphism in $\text{End}_{\calcatname{C}}(\textbf{1})$. Here, $ev$, $coev$, $\beta$ and $\mathds{1}_{A^{*}}$ are respectively the evaluation and coevaluation for $A$, the symmetric braiding, and the identity for $A^{*}$. 

\begin{definition} (Good category)
    A \textit{good category} is an abelian tensor category with finite-dimensional hom-spaces.
\end{definition}
Examples of good categories include $(\catname{s})\catname{Vec}_{\mathds{K}}$, the category of finite-dimensional (super-)vector spaces over $\mathds{K}$, and $\catname{Rep}(G)$, the category of representations of any reductive group $G$. \\

We will need the following properties for good categories.
\begin{proposition} (Corollary 2.3 in \cite{meir-interpolation})\label{trace-of-endomorphism} Let $T$ be an endomorphism in a good category. Then 
\[ \sum_{k\geqslant 0} \text{Tr}(T^k)X^k \quad \in \quad \text{span}_{\mathds{K}} \left\{ \frac{1}{1-\lambda X} : \lambda\in\mathds{K}\right\} \quad \subseteq \quad \mathds{K}[[X]].\]
\end{proposition}

\begin{corollary}\label{trace-of-endomorphisms}
Let $T$ and $S$ be endomorphisms of an object in a good category such that $T \circ S = S \circ T$. Then 
\[ \sum_{k,n\geqslant 0} \text{Tr}(T^k \circ S^n)X^k Y^n \quad \in \quad \text{span}_{\mathds{K}} \left\{ \frac{1}{1-\lambda X}\frac{1}{1-\mu Y} : \lambda,\mu\in\mathds{K}\right\} \quad \subseteq \quad \mathds{K}[[X,Y]].\]
\end{corollary}
\begin{proof}
    By noting that commuting endomorphisms preserve their mutual eigenspaces, the proof of the corollary \ref{trace-of-endomorphisms} can be obtained \textit{mutatis mutandis} from the proof of the above proposition \ref{trace-of-endomorphism}.
\end{proof}

\end{section}

\begin{section}{Interpolation of symmetric monoidal categories}
\label{section-meir-tool}

In this section we review the principal aspects of the tool developed in \cite{meir-interpolation} to construct symmetric monoidal categories from an algebra of diagrams. We restrict ourselves to the intuitive understanding of this tool in order to use it in the remaining sections. 

\begin{subsection}{Algebraic types and closed diagrams}

\begin{definition} (Algebraic type)
    A \textit{type} $t$ is a finite set of boxes, each depicted like the following picture.
    \[ \underbrace{\overbrace{\input{annex/example-x.tikz}}^{p_x}}_{q_x} \]
    Each box has a different label $x$ and specified numbers $p_x$ and $q_x$ in $\mathds{N}$ of outgoing and incoming lines (reading from bottom to top). We require types to have a distinguished box, labelled $\mathds{1}$, with $p_{\mathds{1}}=q_{\mathds{1}}=1$. We often denote a type $t$ by $t=((p_x,q_x))$.
\end{definition}

\begin{example*}
    The type $((1,1), (1,2))$ consists in the following set of boxes.
    \[ \input{LA/box-identity.tikz} \qquad \input{LA/box-l.tikz}  \]
\end{example*} 

\begin{definition} (Diagram)
    Let $t$ be a type. A \textit{diagram} is a combination of (copies of) the boxes in $t$. Such combinations consist of the following: 
    \begin{itemize}
        \item all the boxes in $t$ are diagrams;
        \item tensoring: we can form a diagram by putting two diagrams next to one another;
        \item composition: we can form a new diagram from an old one by linking a incoming line to an outgoing line;
        \item permutation: we can form a new diagram from an old one by performing any permutation of the free incoming lines, and by performing any permutation of the free outgoing lines.
    \end{itemize}
    Furthermore, the distinguished box labelled by $\mathds{1}$ acts as the identity for composition, and only the order of incoming and outgoing lines matters, not the order of the boxes. 
\end{definition}
For example, for a type $((1,1),(2,1),(1,1),(0,2))$, we could form the diagram
\[ \input{annex/tensor-x1-x2-x3-composed.tikz} \]
and we would have for instance the following identities:
\[ 
    \input{annex/permutation-1.tikz} = \input{annex/permutation-2.tikz} \qquad \qquad \input{annex/identity-1.tikz} =  \input{annex/identity-2.tikz}.
\]

\begin{definition} (Algebras $Con$ and $\mathds{K}[X]_{\text{aug}}$) Let $t$ be a type. Define $Con^{p,q}$ to be the $\mathds{K}$-vector space spanned by all diagrams for the type $t$ with a total of $p$ outgoing and $q$ incoming lines, and $Con = \bigoplus_{p,q} Con^{p,q}$ to be the unital algebra of \textit{constructible diagrams}, with multiplication given by tensoring of diagrams and where the identity element is the empty diagram. The commutative subalgebra $\mathds{K}[X]_{\text{aug}}:=Con^{0,0}\subseteq Con$ is the polynomial ring on the set of \textit{closed} and \textit{connected} diagrams (i.e., those diagrams which have no free inputs or ouputs, and which cannot be written as the tensoring of two or more non-empty diagrams). 
\end{definition}

\begin{definition} (Algebraic structure)
    Let $t$ be a type. An \textit{algebraic structure} of type $t$ in a tensor category $\calcatname{C}$ is an object $V$ of $\calcatname{C}$ together with a specified set of morphisms $x_V:V^{\otimes q_x}\rightarrow V^{\otimes p_x}$, called \textit{structure tensors}, one for each box $x$ of the type $t$ ($V^{\otimes 0}\equiv \textbf{1}$ is the unit object of $\calcatname{C}$). We require the identity box to be associated to the identity morphism. Any element in $Con^{p,q}$ can be realised naturally as a string diagram in $\calcatname{C}$, i.e., as a morphism $V^{\otimes q}\rightarrow V^{\otimes p}$, using composition and tensoring of the structure morphisms $x_V$ of $V$, and the duality structure of $V$. We sometimes denote the algebraic structure $V$ by $(V,(x_V))$ to emphasize its structure tensors.
\end{definition}

\begin{definition} (Isomorphism of algebraic structures)
    Let $t$ be a type and $\calcatname{C}$ a tensor category. Two algebraic structures $(V,(x_V))$ and $(W,(x_W))$ of type $t$ in $\calcatname{C}$ are said to be \textit{isomorphic} if there is an isomorphism $u\in\calcatname{C}(V,W)$ such that for all boxes $x$ of $t$, the corresponding structure tensors $x_V:V^{\otimes q_x}\rightarrow V^{\otimes p_x}$ and $x_W:W^{\otimes q_x}\rightarrow W^{\otimes p_x}$ satisfy $x_W \circ u^{\otimes q_x} = u^{\otimes p_x}  \circ x_V$.
\end{definition}

\begin{definition} (Theory)
Let $t$ be a type. A \textit{theory} $\mathcal{T}$ for $t$ consists in a subset $\mathcal{T}\subset\sqcup_{p,q} Con^{p,q}$. Each element of $\mathcal{T}$ is called an \textit{axiom}. 
\end{definition}

\begin{example*}[continued]
    The following set of axioms defines a theory $\mathcal{T}_{\text{LA}}$ (the theory of \textit{Lie algebras}) for the type $((1,1),(1,2))$.
    \[
        \input{LA/axiom-jacobi-1.tikz} + \input{LA/axiom-jacobi-2.tikz} + \input{LA/axiom-jacobi-3.tikz} \quad ; \quad \input{LA/box-l.tikz} + \input{LA/axiom-l-antisymmetric-2.tikz}
    \]
\end{example*}

\begin{definition} (Model of a theory)
    Let $t$ be a type, $\mathcal{T}$ a theory for $t$, and $V$ an algebraic structure of type $t$ in a tensor category $\calcatname{C}$. We say that $V$ is a \textit{model} for $\mathcal{T}$ in $\calcatname{C}$ if all the axioms of $\mathcal{T}$ are realised via $V$ as $0$ morphisms in $\calcatname{C}$.
\end{definition}
For that reason, we will often describe an axiom by an equality between different (linear combinations of) diagrams.

\begin{example*}[continued]
    The simple Lie algebra $\mathfrak{sl}(2)$ is a model for $\mathcal{T}_{\text{LA}}$ in the category of finite-dimensional $\mathds{K}$-vector spaces. The identity box is sent to the identity linear map of $\mathfrak{sl}(2)$, and the other box is sent to the Lie bracket of $\mathfrak{sl}(2)$. Anti-symmetry and Jacobi identity of this Lie bracket ensure that the axioms in $\mathcal{T}_{\text{LA}}$ vanish. 
\end{example*}

\end{subsection}

\begin{subsection}{Meir's good characters}

\begin{definition} \label{definition-character} (Character, algebra of characters)
    Let $t$ be a type. A ring homomorphism $\chi:\mathds{K}[X]_{\text{aug}}\rightarrow \mathds{K}$ is called a \textit{character}. We denote by $\text{Ch}_{\emptyset}$ the set of characters. It has a natural structure of an algebra (cf. section 6 in \cite{meir-interpolation}): for two characters $\chi_1,\chi_2\in \text{Ch}_{\emptyset}$, and any scalar $\lambda\in\mathds{K}$, we define the characters $\chi_1+\chi_2$ (addition), $\chi_1\chi_2$ (product), and $\lambda\chi_1$ (multiplication by a scalar) by their actions on any closed connected diagram $D\in\mathds{K} [X]_{\text{aug}}$:
    \[  
    \left\{
    \begin{array}{lll}
        (\chi_1\chi_2)(D) & = & \chi_1(D)\chi_2(D), \\(\chi_1+\chi_2)(D) & = & \chi_1(D)+\chi_2(D), \\
        (\lambda\chi_1)(D) & = & \lambda(\chi_1(D)). 
    \end{array} 
    \right.
    \]
\end{definition}

Let $t$ be a type and $V$ an algebraic structure of type $t$ in a tensor category $\calcatname{C}$, with unit object $\textbf{1}$. The realisation of $f\in \mathds{K}[X]_{\text{aug}}$ in $\calcatname{C}$ via $V$ gives a morphism in $\text{End}_{\calcatname{C}}(\textbf{1})$, in other words, an element of $\mathds{K}$, and this realisation is a ring homomorphism. Thus, any algebraic structure $V$ of type $t$ defines a character $\chi_V\in\text{Ch}_{\emptyset}$. We say that $V$ affords $\chi_V$. Conversely, as we will see in the next subsection, any character $\chi \in \text{Ch}_{\emptyset}$ can be used to build a tensor category $\calcatname{C}_{\chi}$ containing an object affording $\chi$.

\begin{remark}
    A character $\chi:\mathds{K}[X]_{\text{aug}}\rightarrow \mathds{K}$ is entirely determined by its values on the closed connected diagrams, i.e., specifying a character $\mathds{K}$ is equivalent to specifying a $\mathds{K}$-valued sequence indexed by the set of closed connected diagrams.
\end{remark}

The characters are interesting because they give a first necessary (but not sufficient) condition for two algebraic structures in a tensor category to be isomorphic:
\begin{proposition}
    Let $V,W$ be two isomorphic structures of type $t$ in a tensor category $\calcatname{C}$. Then $\chi_V=\chi_W$.
    \begin{proof}
        This follows directly from the definition of isomorphism of algebraic structures and the fact that closed diagrams do not have any free incoming or outgoing lines. Any closed diagram $f$ is interpreted along $V$ and $W$ as the trace of endomorphisms $f_V$ and $f_W$ in $\calcatname{C}$, related by conjugation by the isomorphism between $V$ and $W$. The trace operation is invariant under such conjugation, thus the closed diagram take the same value, i.e., $\text{Tr}(f_V)=\text{Tr}(f_W)$.
    \end{proof}
\end{proposition}

Given a character $\chi$, some elements in $Con$ might be automatically annihilated by $\chi$, in the following sense: 
\begin{definition} ($\chi$-negligible element)
An element $f \in Con^{p,q}$ is called \textit{$\chi$-negligible} if it is in the radical $rad^{p,q}_{\chi}$ of the $\mathds{K}$-bilinear pairing
\begingroup
\allowdisplaybreaks[0]
\begin{align*}
	pair^{p,q}_{\chi}:Con^{p,q}\times Con^{q,p} & \quad \rightarrow \quad \mathds{K} \\
	\left(\input{annex/pairing-1.tikz}\right) &\quad  \mapsto  \quad \chi \left( \input{annex/pairing-2.tikz}  \right).
\end{align*}
\endgroup
\end{definition}
Intuitively, a diagram will be $\chi$-negligible if any closed diagram in which it appears is sent to $0$ by $\chi$. We define $rad^{p,q}_{\chi}$ to be the radical of $pair^{p,q}_{\chi}$; these radicals are important ingredients in the building of the category $\calcatname{C}_{\chi}$, described in the next section. 
Some characters produce categories $\calcatname{C}_{\chi}$ with additional good properties, similar to the category $\catname{Vec}_{\mathds{K}}$. 

\begin{definition} \label{theorem-udi-good-character} (cf. Theorem 1.1 in \cite{meir-interpolation}) Let $\chi: \mathds{K}[X]_{\text{aug}} \rightarrow \mathds{K}$ be a character. It is called \textit{good} if one of the following equivalent conditions is satisfied.
\begin{itemize}
    \item The character $\chi$ is afforded by an algebraic structure in a good category. 
    \item The category $\mathcal{C}_{\chi}$ is semisimple good.
    \item The category $\mathcal{C}_{\chi}$ satisfies the two properties:
        \begin{itemize}
            \item its hom-spaces are finite-dimensional: $\text{dim }\mathcal{C}_{\chi}(A,B) \leqslant \infty$ for all objects $A,B\in \text{Ob}(\mathcal{C}_{\chi})$;
            \item its nilpotent endomorphisms have vanishing trace: for any object $A$ of $\mathcal{C}_{\chi}$ and any $f\in \mathcal{C}_{\chi}(A,A)$, if $f^r=0$ for some $r>0$, then $\mathrm{Tr}(f)=0$.
        \end{itemize}
\end{itemize}
We denote by $\text{Ch}_{\emptyset}^{\mathcal{G}} \subseteq \text{Ch}_{\emptyset}$ the set of good characters. 
We call a sequence \textit{good} if its associated character is good.
\end{definition}

\begin{proposition} (Section 6 in \cite{meir-interpolation})
    The set of good characters $\text{Ch}_{\emptyset}^{\mathcal{G}}$ forms a subalgebra of the algebra of characters $\text{Ch}_{\emptyset}$.
\end{proposition}

\end{subsection}

\begin{subsection}{The categories $\calcatname{C}_{\chi}$}

Given a type $t$ and a theory $\mathcal{T}$ for this type, we can construct an additive $\mathds{K}$-linear rigid symmetric monoidal category $\mathcal{C}_{\text{univ}}^{\mathcal{T}}$. This category is universal in the following sense:
\begin{proposition} \label{proposition-universal-property} (Proposition 3.5 in \cite{meir-interpolation})
    Let $\mathcal{D}$ be a tensor category. The isomorphism classes of symmetric monoidal $\mathds{K}$-linear functors $F: \mathcal{C}_{\text{univ}}^{\mathcal{T}} \rightarrow \mathcal{D}$ are in one-to-one correspondence with isomorphism classes of dualizable models for $\mathcal{T}$ in $\mathcal{D}$.
\end{proposition}
The categories $\calcatname{C}_{\chi}$ are built from this universal category. We follow section 3 in \cite{meir-interpolation}. We start by defining a monoidal category $\mathcal{C}_0$.

\begin{itemize}
    \item Objects of $\mathcal{C}_0$ are symbols $V^{a,b}$ with $a,b\in \mathds{N}$. We can think of them as $V^{\otimes a}\otimes (V^*)^{\otimes b}$.
    \item $\mathcal{C}_0(V^{a,b}, V^{c,d}) = Con^{c+b, d+a}$.
    \item Composition of morphisms is given by the $\mathds{K}$-linear map:
    \begin{align*}
    	\mathcal{C}_0(V^{a,b}, V^{c,d})\otimes \mathcal{C}_0(V^{c,d}, V^{e,f}) & \quad \rightarrow \quad \mathcal{C}_0(V^{a,b}, V^{e,f}) \\
    	\input{annex/composition-c0-2.tikz} \circ \input{annex/composition-c0-1.tikz} & \quad =  \quad \input{annex/composition-c0.tikz} .
    \end{align*}
\end{itemize}
Composition is associative, and there is an identity morphism $\mathds{1}_{V^{a,b}} \in \mathcal{C}_0(V^{a,b}, V^{a,b})$:
\[
    \mathds{1}_{V^{a,b}} \quad := \quad \input{annex/identity-c0.tikz}, \quad \text{ where } \quad \input{annex/identity-c0-Va.tikz} .
\]
We can see $\mathcal{C}_0$ as a rigid symmetric monoidal category with unit object $\textbf{1}=V^{0,0}$ by defining:
\begin{align*} 
    V^{a,b}\otimes V^{c,d} & \quad = \quad V^{a+c,b+d} \\
    \input{annex/tensor-c0-1.tikz} \otimes \input{annex/tensor-c0-2.tikz} & \quad  =  \quad \input{annex/tensor-c0.tikz} .
\end{align*}
An object $V^{a,b}$ has dual object $V^{b,a}$. The evaluation and coevaluation morphisms are both given by the same diagram as the identity on $V^{a,b}$, but interpreted in different hom-spaces. The symmetric braiding is given by the morphisms:
\[
    \beta_{V^{a,b}\otimes V^{c,d}} \quad := \quad \input{annex/symmetric-c0.tikz} \quad : \quad V^{a,b} \otimes V^{c,d} \quad \rightarrow \quad V^{c,d} \otimes V^{a,b} .
\]

Note that $\mathds{K}[X]_{\text{aug}}=\text{End}_{\mathcal{C}_0}(\textbf{1})$.

\begin{definition} (Tensor ideal) Let $\mathcal{D}$ be a $\mathds{K}$-linear symmetric monoidal category. A \textit{tensor ideal} $\mathfrak{I}$ of $\mathcal{D}$ is a collection of subspaces $\mathfrak{I}(A,B) \subseteq \mathcal{D}(A,B)$ for every objects $A$ and $B$ of $\mathcal{D}$ satisfying the following two conditions. For arbitrary objects $A,B,C$ and $D$ in $\mathcal{D}$: 
\begin{itemize}
    \item if $f\in \mathcal{D}(A,B)$, $g\in \mathfrak{I}(B,C)$, $h\in \mathcal{D}(C,D)$, then $h\circ g \circ f\in\mathfrak{I}(A,D)$;
    \item if $f\in \mathcal{D}(A,B)$, $g\in \mathfrak{I}(C,D)$, then $f\otimes g \in \mathfrak{I}(A\otimes C,B\otimes D)$.
\end{itemize}

\end{definition}

\begin{definition} (Quotient category)
    Let $\mathcal{D}$ be a $\mathds{K}$-linear rigid symmetric monoidal category, and $\mathfrak{I}$ a tensor ideal of $\mathcal{D}$. The \textit{quotient category} $\mathcal{D}/\mathfrak{I}$ is the category whose objects are the same as the ones of $\mathcal{D}$, and whose hom-spaces $(\mathcal{D}/\mathfrak{I})(A,B)$ are the quotient spaces $\mathcal{D}(A,B)/\mathfrak{I}(A,B)$. It is again a $\mathds{K}$-linear rigid symmetric monoidal category.
\end{definition}

The universal category $\mathcal{C}_{\text{univ}}^{\mathcal{T}}$ for the theory $\mathcal{T}$ is defined as the quotient by a certain tensor ideal $\mathfrak{I}_{\mathcal{T}}$ of the additive envelope of $\mathcal{C}_0$. This tensor ideal is the tensor ideal generated by the elements in $\mathcal{T}$, seen as morphisms in the additive envelope of $\mathcal{C}_0$. We thus have in particular that $\mathcal{C}_{\text{univ}}^{\mathcal{T}}(\textbf{1},\textbf{1}) = \mathds{K}[X]_{\text{aug}} / \mathfrak{I}_{\mathcal{T}}(\textbf{1},\textbf{1}) =: \mathds{K}[X]_{\text{aug}}^{\mathcal{T}}$. \\

We define $\text{Ch}_{\mathcal{T}}$ to be the set of ring homomorphisms $\{ \chi:\mathds{K}[X]_{\text{aug}}^{\mathcal{T}}\rightarrow \mathds{K} \}\subseteq \text{Ch}_{\emptyset}$, and $\text{Ch}_{\mathcal{T}}^{\mathcal{G}}:=\text{Ch}_{\mathcal{T}}\cap \text{Ch}_{\emptyset}^{\mathcal{G}}$ to be its subset of good characters. For some theories $\mathcal{T}$, $\text{Ch}_{\mathcal{T}}$ also forms an algebra in its own right, as we will see in section \ref{section-good-knfrob}. \\

Given a character $\chi: \mathds{K}[X]_{\text{aug}}^{\mathcal{T}} \rightarrow \mathds{K}$, the collection of $\chi$-negligible morphisms generates a tensor ideal $\mathfrak{N}_{\chi}$ of $\mathcal{C}_{\text{univ}}^{\mathcal{T}}$. We define $\mathcal{C}_{\chi}$ to be the Karoubi envelope of $\mathcal{C}_{\text{univ}}^{\mathcal{T}}/\mathfrak{N}_{\chi}$. The category $\mathcal{C}_{\chi}$ is $\mathds{K}$-linear, rigid symmetric monoidal, additive, Karoubi closed, and satisfies $\text{End}_{\calcatname{C}_{\chi}}(\textbf{1}) \cong \mathds{K}$. The object $V=V^{1,0}$ is a model for $\mathcal{T}$ in $\mathcal{C}_{\chi}$. \\

Finally, the categories of representations of groups stated in the introduction fit this setup by the following proposition. Let $t$ be a type, and $\mathcal{T}$ a theory for this type. Let $V$ be a model for $\mathcal{T}$ in the category $\catname{Vec}_{\mathds{K}}$, and $d=\text{dim}(V)$. Choosing a basis for $V$, the algebraic structure of $V$ takes the form of a tuple of numbers, namely the coordinates of each structure tensor of $V$ expressed in the chosen basis. We interpret this tuple as a point in an affine space. Changing basis, i.e., acting with $GL_d(\mathds{K})$, draws an orbit in this affine space. 
\begin{proposition} (Proposition 5.2 in \cite{meir-interpolation}) \label{proposition-rep-automorphism-group}
If $V$ has closed $GL_d$-orbit, then we have an equivalence 
 of categories $\calcatname{C}_{\chi_{V}}\simeq \catname{Rep}(\text{Aut}(V))$, where the latter is the category of representations of the automorphism group of~$V$.
\end{proposition}

\end{subsection}

\end{section}
\begin{section}{Knowledgeable Frobenius algebras and the category $\catname{2Cob}_{O-C}$}
\label{section-open-closed}

Lauda and Pfeiffer \cite{lauda-pfeiffer-1} showed that for any tensor category $\mathcal{D}$, the open-closed oriented TQFTs with values in $\mathcal{D}$ are equivalent to the so-called knowledgeable Frobenius algebra objects in $\mathcal{D}$ (cf. the third line of table \ref{table:1}). We will focus first on the definition of the theory of knowledgeable Frobenius algebras, and then turn to the bordism category $\catname{2Cob}_{O-C}$, the category of oriented open-closed $2$-dimensional cobordisms.

\begin{subsection}{The theory of Frobenius algebras}

Models of knowledgeable Frobenius algebras consists of pairs of \textit{Frobenius algebras}. We thus start by describing the theory $\mathcal{T}_{\text{FA}}$ of Frobenius algebras. It is of type 
\[ t_{\text{FA}}=((1,1),(1,2),(1,0),(2,1),(0,1)), \] 

namely, it consists in the following set of boxes: 
\[ 
\begin{matrix}
    \begin{tikzpicture}
	\begin{pgfonlayer}{nodelayer}
		\node [style=box] (0) at (0, 0) {$\mathds{1}$};
		\node [style=none] (2) at (0, -1) {};
		\node [style=none] (4) at (0, -0.25) {};
		\node [style=none] (5) at (0, 0.25) {};
		\node [style=none] (6) at (0, 1) {};
	\end{pgfonlayer}
	\begin{pgfonlayer}{edgelayer}
		\draw (2.center) to (4.center);
		\draw (5.center) to (6.center);
	\end{pgfonlayer}
\end{tikzpicture}
 & \begin{tikzpicture}
	\begin{pgfonlayer}{nodelayer}
		\node [style=box] (0) at (0, 0) {$\nabla$};
		\node [style=none] (1) at (-0.25, -1) {};
		\node [style=none] (2) at (0, 0.25) {};
		\node [style=none] (3) at (-0.25, -0.25) {};
		\node [style=none] (4) at (0, 1) {};
		\node [style=none] (5) at (0.25, -0.25) {};
		\node [style=none] (6) at (0.25, -1) {};
	\end{pgfonlayer}
	\begin{pgfonlayer}{edgelayer}
		\draw (2.center) to (4.center);
		\draw (6.center) to (5.center);
		\draw (3.center) to (1.center);
	\end{pgfonlayer}
\end{tikzpicture}
 & \begin{tikzpicture}
	\begin{pgfonlayer}{nodelayer}
		\node [style=box] (0) at (0, 0) {$u$};
		\node [style=none] (2) at (0, 0.25) {};
		\node [style=none] (4) at (0, 1) {};
	\end{pgfonlayer}
	\begin{pgfonlayer}{edgelayer}
		\draw (2.center) to (4.center);
	\end{pgfonlayer}
\end{tikzpicture}
 & \begin{tikzpicture}
	\begin{pgfonlayer}{nodelayer}
		\node [style=box] (0) at (0, 0) {$\Delta$};
		\node [style=none] (23) at (0, -1) {};
		\node [style=none] (24) at (-0.25, 0.25) {};
		\node [style=none] (25) at (0, -0.25) {};
		\node [style=none] (26) at (-0.25, 1) {};
		\node [style=none] (27) at (0.25, 1) {};
		\node [style=none] (28) at (0.25, 0.25) {};
	\end{pgfonlayer}
	\begin{pgfonlayer}{edgelayer}
		\draw (24.center) to (26.center);
		\draw (28.center) to (27.center);
		\draw (25.center) to (23.center);
	\end{pgfonlayer}
\end{tikzpicture}
 & \begin{tikzpicture}
	\begin{pgfonlayer}{nodelayer}
		\node [style=box] (0) at (0, 0) {$\epsilon$};
		\node [style=none] (2) at (0, -1) {};
		\node [style=none] (4) at (0, -0.25) {};
	\end{pgfonlayer}
	\begin{pgfonlayer}{edgelayer}
		\draw (2.center) to (4.center);
	\end{pgfonlayer}
\end{tikzpicture}
 \\
    \text{(Identity)} & \text{(Product)} & \text{(Unit)} & \text{(Coproduct)} & \text{(Counit)} \\
\end{matrix}.
\]

The theory $\mathcal{T}_{\text{FA}}$ can be described by the following axioms.
\[
    \begin{matrix}
        \begin{matrix}
            \input{Frob/axioms/unit-1.tikz} & = & \input{Frob/axioms/unit-0.tikz} & = & \input{Frob/axioms/unit-2.tikz} 
        \end{matrix}
        & \quad &
        \begin{matrix}
            \input{Frob/axioms/counit-1.tikz} & = & \input{Frob/axioms/counit-0.tikz} & = & \input{Frob/axioms/counit-2.tikz} 
        \end{matrix}
        \\
        \text{(Unitality)} & & \text{(Counitality)}
    \end{matrix}
\]
\[
    \begin{matrix}
        \begin{matrix}
            \input{Frob/axioms/assoc-1.tikz} & = &  \input{Frob/axioms/assoc-2.tikz} 
        \end{matrix}
        & \quad &
        \begin{matrix}
            \input{Frob/axioms/coassoc-1.tikz} & = &  \input{Frob/axioms/coassoc-2.tikz} 
        \end{matrix}
        \\
        \text{(Associativity)} & & \text{(Coassociativity)}
    \end{matrix}
\]
\[
    \begin{matrix}
        \begin{matrix}
            \input{Frob/axioms/frob-1.tikz} & = & \input{Frob/axioms/frob-0.tikz} & = & \input{Frob/axioms/frob-2.tikz} 
        \end{matrix}
        \\
        \text{(Frobenius relation)}
    \end{matrix}
\]\\

A model for the theory of Frobenius algebras in a tensor category $\calcatname{D}$ (with unit object $\textbf{1}$) is an object $V$ of $\calcatname{D}$ together with morphisms $\nabla:V\otimes V\rightarrow V$,  $u:\textbf{1}\rightarrow V$, $\Delta:V\rightarrow V\otimes V$, and $\epsilon: V\rightarrow \textbf{1}$, satisfying all the relations listed above: (co)unitality, (co)associativity, and the Frobenius relation. Note that any Frobenius algebra possesses a non-degenerate pairing $\epsilon\circ\nabla:V\otimes V\rightarrow \textbf{1}$ with associated copairing $\Delta\circ u :\textbf{1}\rightarrow V\otimes V$ (the non-degeneracy comes from the Frobenius relation). The counit $\epsilon$ is also called the \textit{Frobenius form}. \\

In the category $\catname{Vec}_{\mathds{K}}$, a Frobenius algebra structure can be equivalently specified by giving a non-degenerate linear form $\epsilon: A\rightarrow \mathds{K}$ to a unital associative algebra $(A,\nabla,u)$, in the sense that $\epsilon\circ\nabla$ is a non-degenerate bilinear pairing (cf. section 2.2 in \cite{kock-book}). \\

The following families of models for $\mathcal{T}_{\text{FA}}$ in $\catname{Vec}_{\mathds{K}}$ are of particular interest to us.

\begin{example} \textbf{The Frobenius algebra $A_{\alpha,\delta}(n)$, with $n\in \mathds{N}$, $\alpha \in \mathds{K}^*$, and $\delta \in \mathds{K}$.} We consider the category $\catname{Vec}_{\mathds{K}}$. We denote by $A_{\alpha,\delta}(n)$ the $(n+2)$-dimensional vector space spanned by $\{1, a, a_{1}, ..., a_{n}\}$ and having the following Frobenius algebra structure: 
\begin{itemize}
    \item \underline{Unit} $u:\mathds{K}\rightarrow A_{\alpha,\delta}(n)$: $\qquad u(1) = 1$.
    \item \underline{Product} $\nabla:A_{\alpha,\delta}(n)\otimes A_{\alpha,\delta}(n)\rightarrow A_{\alpha,\delta}(n)$:
        \[
            \nabla(a_i\otimes a) = \nabla(a\otimes a_i) = \nabla(a\otimes a) = 0,
            \qquad 
            \nabla(a_i\otimes a_j) = 
            \left\{ 
            \begin{array}{l}
                a \text{ if } i=j \\
                0 \text{ else}
            \end{array} 
            \right..
        \]
    \item \underline{Counit} $\epsilon: A_{\alpha,\delta}(n)\rightarrow \mathds{K} $: $\qquad \epsilon(1) = \delta, \qquad \epsilon(a)=\alpha, \qquad \epsilon(a_i)=0$.
    \item \underline{Coproduct} $\Delta:A_{\alpha,\delta}(n)\rightarrow A_{\alpha,\delta}(n)\otimes A_{\alpha,\delta}(n)$: 
        \[
            \Delta(1) = \frac{1}{\alpha}\left(1\otimes a + a \otimes 1 + \sum_{i=1}^{n} a_i \otimes a_{i}\right) - \frac{\delta}{\alpha^2} a\otimes a,
        \]
        \[ 
            \Delta(a) = \frac{1}{\alpha}a\otimes a, \qquad
            \Delta(a_i) = \frac{1}{\alpha}\left(a_i\otimes a + a\otimes a_i\right).
        \]
        
\end{itemize}

The Frobenius algebra $A_{\alpha,\delta}(n)$ is a \textit{commutative} Frobenius algebra, in the sense that its multiplication is commutative: $\forall x,y\in A_{\alpha,\delta}(n)$, we have $\nabla(x\otimes y)=\nabla(y\otimes x)$. More generally, we have: 
    
\end{example}

\begin{example} \textbf{The Frobenius algebra $A_{\alpha,\delta}(V)$}, with $\alpha \in \mathds{K}^*$, $\delta \in \mathds{K}$, and $V$ an object with a non-degenerate (symmetric) pairing $c:V\otimes V\rightarrow \textbf{1}$ in an additive tensor category with unit object $\textbf{1}$ (cf. example 2.1 in \cite{kok}). We define $A_{\alpha,\delta}(V):=U\oplus V \oplus N$, where $U\xleftarrow{u} \textbf{1} \xrightarrow{n} N$ are isomorphisms. The Frobenius algebra structure is given by:  
\begin{itemize}
    \item \underline{Unit} $\textbf{1}\xrightarrow{u} U$.
    \item \underline{Product} $\nabla|_{V\otimes V}= V\otimes V \xrightarrow{c}\textbf{1}\xrightarrow{n}N$, $\qquad \nabla|_{N\otimes V}$, $\nabla|_{V\otimes N}$, and $\nabla|_{N\otimes N}$ are $0$.
    \item \underline{Counit} $\epsilon|_{U}= \delta u^{-1}$, $\qquad \epsilon|_{V}= 0$, $\qquad \epsilon|_{N}= \alpha n^{-1}$.
    \item \underline{Coproduct} Writing $\eta:\textbf{1}\rightarrow V\otimes V$ to be the copairing of $c$: 
        \begin{multline*}
            \Delta|_U = 
            \frac{1}{\alpha}\left[ \vphantom{\frac{\delta}{\alpha^2}}
            (U\cong U\otimes \textbf{1} \xrightarrow{\mathds{1}\otimes n} U\otimes N )
            + (U\cong \textbf{1} \otimes U \xrightarrow{n\otimes\mathds{1}} N\otimes U ) \right.
            \\ + \left. (U\xrightarrow{u^{-1}} \textbf{1} \xrightarrow{\eta} V\otimes V )
            \right] 
            - \frac{\delta}{\alpha^2} \left(U\xrightarrow{u^{-1}} \textbf{1} \cong  \textbf{1} \otimes \textbf{1} \xrightarrow{n\otimes n} N\otimes N \right),
        \end{multline*}
        \[ 
            \Delta|_V = \frac{1}{\alpha}\left[ (V\cong \textbf{1}\otimes V \xrightarrow{n\otimes \mathds{1}} N\otimes V)
            + (V\cong V\otimes \textbf{1} \xrightarrow{\mathds{1}\otimes n} V\otimes N) \right],
        \]
        \[
            \Delta|_N = \frac{1}{\alpha}\left(N\cong \textbf{1}\otimes N \xrightarrow{n\otimes\mathds{1}}N\otimes N\right).
        \]
\end{itemize}
If $c$ is symmetric, then $A_{\alpha,\delta}(V)$ is commutative. If $V$ has dimension $d$, then $A_{\alpha,\delta}(V)$ has dimension $d+2$. 
\end{example}

\begin{example}\textbf{The Frobenius algebra $F_{\alpha}$, with $\alpha\in\mathds{K}^*$.} We consider the category $\catname{Vec}_{\mathds{K}}$. We denote by $F_{\alpha}$ the $1$-dimensional vector space $\mathds{K}$ (with basis $\{1\})$ with the following Frobenius algebra structure:
\begin{itemize}
    \item \underline{Unit} $u:\mathds{K}\rightarrow \mathds{K}$: $\qquad u(1) = 1$.
    \item \underline{Product} $\nabla:\mathds{K}\otimes \mathds{K}\rightarrow \mathds{K}$: $\qquad \nabla(1 \otimes 1) = 1$.
    \item \underline{Counit} $\epsilon: \mathds{K}\rightarrow \mathds{K} $: $\qquad \epsilon(1) = \alpha$.
    \item \underline{Coproduct} $\Delta:\mathds{K}\rightarrow \mathds{K}\otimes \mathds{K}$: $\qquad \Delta(1)=\frac{1}{\alpha}~1\otimes 1$.
\end{itemize}
This forms also a commutative Frobenius algebra. More generally:  
\end{example}

\begin{example}\textbf{The Frobenius algebra $F_{\alpha}(n)$, with $\alpha\in\mathds{K}^*$ and $n\in\mathds{N}$.} We consider the category $\catname{Vec}_{\mathds{K}}$, and in it the algebra of $n\times n$ matrices with $\mathds{K}$-valued entries. We take $\{e_{ij} : i,j=1,...,n\}$ as a basis, where $e_{ij}$ has entries equal to $0$ except for the component $(i,j)$ where it is equal to $1$. Together with the following structure, this forms a Frobenius algebra which we denote $F_{\alpha}(n)$.
\begin{itemize}
    \item \underline{Unit} $u:\mathds{K}\rightarrow F_{\alpha}(n)$: the identity $n\times n$ matrix.
    \item \underline{Product} $\nabla:F_{\alpha}(n)\otimes  F_{\alpha}(n)\rightarrow F_{\alpha}(n)$: the usual matrix product.
    \item \underline{Counit} $\epsilon: F_{\alpha}(n)\rightarrow \mathds{K} $: $\qquad \epsilon(e_{ij}) = \left\{\begin{array}{ll}\alpha & \text{ if } i=j \\ 0 & \text{ else} \end{array}\right.$.
    \item \underline{Coproduct} $\Delta:F_{\alpha}(n)\rightarrow F_{\alpha}(n)\otimes F_{\alpha}(n)$: $\qquad \Delta(e_{ij})=\frac{1}{\alpha}\sum_{l=1}^n e_{il}\otimes e_{lj}$.
\end{itemize}
For $n=0$, this gives the trivially null structure. For $n=1$, we have the equality $F_{\alpha}(1)=F_{\alpha}$. For $n\geqslant 2$, $F_{\alpha}(n)$ is not commutative, but is nonetheless \textit{symmetric}, in the sense that the pairing $\epsilon\circ\nabla: F_{\alpha}(n)\otimes F_{\alpha}(n)\rightarrow \mathds{K}$ is symmetric: $\forall x,y\in F_{\alpha}(n)$, we have $(\epsilon\circ\nabla)(x\otimes y)=(\epsilon\circ\nabla)(y\otimes x)$. There exists a generalisation of this family which allows, in some sense, for $n$ to take general values in $\mathds{K}$. This more general family is described in appendix \ref{section-annex-gl}.
\end{example}

The next proposition shows that any symmetric Frobenius algebra structure on the algebra of $n\times n$ matrices is on the form $F_{\alpha}(n)$ for some $\alpha\in\mathds{K}^*$.

\begin{proposition} \label{proposition-relation-symmetric-FA}
    Let $(A,\nabla,u)$ be a unital associative algebra in $\catname{Vec}_{\mathds{K}}$ with center $Z(A)$, and with two symmetric Frobenius algebra structures specified by symmetric Frobenius forms $\epsilon, \epsilon':A\rightarrow \mathds{K}$. Then $\exists a\in Z(A)$ such that $\forall x\in A, \epsilon(x)=\epsilon'(a x)$.
\end{proposition}
\begin{proof}
    Let $\gamma':\mathds{K}\rightarrow A\otimes A$ be the copairing associated to the pairing $\epsilon'\circ\nabla$. The element we are looking for is $a:=\gamma'\circ (\mathds{1}_A\otimes \epsilon)$.
\end{proof}

\end{subsection}

\begin{subsection}{The theory of knowledgeable Frobenius algebras}
\label{subsection-KFA}
    
The theory of KFAs is of type 
\[ t_{\text{KFA}}=((1,1))\sqcup((1,1),(1,1))\sqcup t_{\text{FA}}\sqcup t_{\text{FA}}, \]
namely, it contains two copies of Frobenius algebra boxes, which we distinguish by adding a subscript $I$ for the boxes of one set, called the \textit{open sector}, and a subscript $S^1$ for the boxes of the other set, called the \textit{closed sector}. The three new boxes are the following:
\[ 
\begin{matrix}
    \input{KFrob/id.tikz} & \input{KFrob/iota.tikz} & \input{KFrob/iota-star.tikz} \\
    \text{(Identity)} & \text{(Zipper)} & \text{(Cozipper)} \\
\end{matrix}.
\]

In addition to the axioms of $\mathcal{T}_{\text{FA}}$ for each of the copies of $t_{\text{FA}}$, the theory $\mathcal{T}_{\text{KFA}}$ also consists of:
\[
    \begin{matrix}
        \begin{matrix}
            \begin{matrix}
                \input{KFrob/id.tikz} & = & \input{KFrob/I-id.tikz}  + \input{KFrob/S-id.tikz} 
            \end{matrix}
            & \quad &
            \begin{matrix}
                \input{KFrob/axioms/disjoint-1.tikz} & = & 0 & = & \input{KFrob/axioms/disjoint-2.tikz} 
            \end{matrix}
        \end{matrix} \\
        \text{(A knowledgeable Frobenius algebra is made of two disjoint Frobenius algebras)}
    \end{matrix}
\]
\[
    \begin{matrix}
        \begin{matrix}
            \begin{matrix}
                \input{KFrob/axioms/com-1.tikz} & = & \input{KFrob/axioms/com-2.tikz}  
            \end{matrix}
            & \quad &
            \begin{matrix}
                \input{KFrob/axioms/cocom-1.tikz} & = & \input{KFrob/axioms/cocom-2.tikz}  
            \end{matrix} \\
            \text{(Commutativity)} & & \text{(Cocommutativity)}
        \end{matrix} \\
        \text{(The closed sector forms a commutative Frobenius algebra)}
    \end{matrix}
\]
\[
    \begin{matrix}
        \begin{matrix}
            \begin{matrix}
                \input{KFrob/axioms/sym-1.tikz} & = & \input{KFrob/axioms/sym-2.tikz}  
            \end{matrix}
            & \quad &
            \begin{matrix}
                \input{KFrob/axioms/cosym-1.tikz} & = & \input{KFrob/axioms/cosym-2.tikz}  
            \end{matrix} \\
            \text{(Symmetricity)} & & \text{(Cosymmetricity)}
        \end{matrix} \\
        \text{(The open sector forms a symmetric Frobenius algebra)}
    \end{matrix}
\]
\[
    \begin{matrix}
        \begin{matrix}
            \input{KFrob/axioms/compatibilite-iota-3.tikz} & = & \input{KFrob/axioms/compatibilite-iota-2.tikz} & = & \input{KFrob/axioms/compatibilite-iota-1.tikz} 
        \end{matrix}
        & \quad &
        \begin{matrix}
            \input{KFrob/axioms/compatibilite-iota-star-3.tikz} & = & \input{KFrob/axioms/compatibilite-iota-star-2.tikz} & = & \input{KFrob/axioms/compatibilite-iota-star-1.tikz} 
        \end{matrix} \\
        \text{(}\iota\text{ goes from closed to open sector)} & & \text{(}\iota^*\text{ goes from open to closed sector)}
    \end{matrix} 
\]

\[
    \begin{matrix}
        \begin{matrix} 
            \input{KFrob/axioms/hom-1.tikz} & = & \begin{tikzpicture}
	\begin{pgfonlayer}{nodelayer}
		\node [style=box] (0) at (0, 0) {$u_I$};
		\node [style=none] (2) at (0, 0.25) {};
		\node [style=none] (4) at (0, 1) {};
	\end{pgfonlayer}
	\begin{pgfonlayer}{edgelayer}
		\draw (2.center) to (4.center);
	\end{pgfonlayer}
\end{tikzpicture}
 & \quad \quad & \input{KFrob/axioms/hom-3.tikz} & = & \input{KFrob/axioms/hom-4.tikz}
        \end{matrix} \\
        \text{(} \iota \text{ is an algebra homomorphism)}
    \end{matrix}
\]
\[
    \begin{matrix}
        \begin{matrix}
            \input{KFrob/axioms/dual-1.tikz} & = & \input{KFrob/axioms/dual-2.tikz}
        \end{matrix} 
        & \quad & 
        \begin{matrix}
            \input{KFrob/axioms/knowledge-1.tikz} & = & \input{KFrob/axioms/knowledge-2.tikz}
        \end{matrix}
         & \quad & 
        \begin{matrix}
            \input{KFrob/axioms/cardy-1.tikz} & = & \input{KFrob/axioms/cardy-2.tikz}
        \end{matrix} \\
        \text{(} \iota^* \text{ is dual to } \iota \text{)} & & \text{("Knowledge" about the center)} & & \text{(Cardy relation)}
    \end{matrix}
\]

A model for the theory $\mathcal{T}_{\text{KFA}}$ will thus be an object in a tensor category $\calcatname{D}$ together with two "disjoint" Frobenius algebra structures talking to one another via specified zipper and cozipper endomorphisms, satisfying the above axioms. 

\begin{remark} \label{remark-idempotent-completion} If $\calcatname{D}$ is idempotent-complete, such as $\catname{Vec}_{\mathds{K}}$, it is often useful to think of a model for $\mathcal{T}_{\text{KFA}}$ as a quadruple $(V_I, V_{S^1}, \iota, \iota^*)$ where $V_I$ is a symmetric Frobenius algebra object in $\calcatname{D}$, $V_{S^1}$ is a commutative Frobenius algebra object in $\calcatname{D}$, and $\iota : V_{S^1} \rightleftarrows V_I : \iota^*$ are morphisms in $\calcatname{D}$ satisfying all the relations above. For that reason, we will also refer to any such quadruple in any tensor category $\calcatname{D}$ as a knowledgeable Frobenius algebra (we can always go in the additive closure of $\calcatname{D}$). 
\end{remark}

In light of this remark, we will now consider interchangeably the category $\calcatname{C}_{\text{univ}}^{\mathcal{T}_{\text{KFA}}}$ and its idempotent completion $(\calcatname{C}_{\text{univ}}^{\mathcal{T}_{\text{KFA}}})^{\text{IC}}$. The generating object $V$ decomposes into the direct sum of two objects $V_I\oplus V_{S^1}$ encapsulating the open and the closed sectors. 

\begin{example} Every commutative Frobenius algebra object $V$ in a tensor category $\calcatname{D}$ can be trivially upgraded to become a KFA object in $\calcatname{D}$ by realising the boxes of the open sector (including $\mathds{1}_I$) as the $0$ morphisms of the appropriate hom-spaces. If the tensor category has a $\textbf{0}$ object, it is equivalently thought of as taking the KFA to be the quadruple $(V_I,V_{S^1},\iota,\iota^*)=(\textbf{0},V,0,0)$.
\end{example}

\begin{example} \label{example-semisimple} \textbf{A semisimple KFA.} Consider the quadruple $(F_{\alpha}(n),F_{\alpha^2}, \iota, \iota^*)$ where $\alpha\in \mathds{K}^*$, $n\in\mathds{N}$, and the zipper and cozipper are given by:
\begin{itemize}
    \item \underline{Zipper} $\iota : F_{\alpha^2} \rightarrow F_{\alpha}(n)$: $\qquad \iota(1)=\sum_{i=1}^n e_{ii}$.
    \item \underline{Cozipper} $\iota^* : F_{\alpha}(n) \rightarrow F_{\alpha^2}$: $\qquad \iota^*(e_{ij})=\left\{\begin{array}{ll}\frac{1}{\alpha} & \text{ if i=j} \\ 0 & \text{ else} \end{array}\right.$.
\end{itemize}
This defines a KFA in the category $\catname{Vec}_{\mathds{K}}$. As before, this example can be generalised so as to allow more values for $n$, for instance $n\in\mathds{K}$ (cf. appendix \ref{section-annex-gl}).
\end{example}

\begin{example} \label{example-non-semisimple} \textbf{A non-semisimple KFA.} Let $p\in\mathds{N}, q\in\mathds{N}^*, \alpha\in\mathds{K}^*, \delta,\sigma\in\mathds{K}$, and $\beta=\frac{\alpha^2}{p+2}$. Consider the quadruple $(V_I,V_{S^1},\iota,\iota^*)$ where $V_I=A_{\alpha,\delta}(p)$ has basis $\{1, a, a_{1}, ..., a_{p}\}$; $V_{S^1} = A_{\beta,\sigma}(q)$ has basis $\{1, b, b_{1}, ..., b_{q}\}$; and where the zipper and cozipper are given by:
\begin{itemize}
    \item \underline{Zipper} $\iota : V_{S^1} \rightarrow V_I$: $\qquad \iota(1)=1$, $\qquad \iota(b)=0$, $\qquad \iota(b_i)=\left\{\begin{array}{ll}a & \text{ if i=q} \\ 0 & \text{ else} \end{array}\right.$.
    \item \underline{Cozipper} $\iota^* : V_I \rightarrow V_{S^1}$: 
    $\quad \iota^*(1)=\frac{1}{\beta}(\delta b + \alpha b_q), \qquad \iota^*(a)= \frac{\alpha}{\beta} b, \qquad \iota^*(a_i)=0.$
\end{itemize}
More generally: 
\end{example}

\begin{example} \label{example-non-semisimple-general} Let $V$ and $W$ be objects endowed with non-degenerate symmetric pairings $c_V$ and $c_W$ in an additive tensor category. Let $\alpha,\beta\in\mathds{K}^*, \delta,\sigma\in\mathds{K}$, and $\gamma \in \mathds{K}$ such that $\gamma^2=(2+\text{dim}V)\frac{\beta}{\alpha^2}$. Consider the quadruple $(V_I,V_{S^1},\iota,\iota^*)$ where $V_I=A_{\alpha,\delta}(V)=U_I \oplus V \oplus N_I$; $V_{S^1} = A_{\beta,\sigma}(W')= U_S\oplus W'\oplus N_S$ with $W':=W\oplus \textbf{1}$ is endowed with the non-degenerate symmetric pairing $c_W'=c_W\oplus 1$ (the latter morphism being the non-degenerate symmetric pairing on $\textbf{1}$ given by the natural isomorphism $\textbf{1}\otimes \textbf{1}\cong \textbf{1}$); and where the zipper and cozipper are given by:
\begin{itemize}
    \item \underline{Zipper} $V_{S^1} \xrightarrow{\iota} V_I = (U_S \cong U_I) + \gamma (\textbf{1}\cong N_I)$,
    \item \underline{Cozipper} $V_I \xrightarrow{\iota^*} V_{S^1} = \frac{1}{\beta}\left[ \delta(U_I\cong N_S) + \alpha(N_I\cong N_S) + \alpha \gamma (U_I\cong \textbf{1}) \right]$,
\end{itemize}
where the isomorphisms involved correspond to the structure morphisms in the definitions of $V_I$ and $V_{S^1}$.
Note that if $\text{dim}V=-2$, then $\gamma=0$, and we can construct the KFA above using simply $W$ instead of $W'$.
\end{example}

\begin{proposition} \label{proposition-sum-product-KFAs} (Sum and product of KFAs)
    Let $(V,(x))=(V_I,V_{S^1}, (x))$ and $(W,(y))=(W_I,W_{S^1},(y))$ be models of KFAs in an idempotent-complete tensor category $\calcatname{D}$ (we have regrouped all structure tensors of each model under the labels $(x)$ and $(y)$ for simplicity). Then, their \textit{sum} and \textit{product}, defined as follow, are also models of KFAs in $\calcatname{D}$:
    \begin{itemize}
        \item $(V_I,V_{S^1}, (x))+(W_I,W_{S^1}, (y))=(V_I\oplus W_I,V_{S^1}\oplus W_{S^1}, (x\oplus y))$,
        \item $(V_I,V_{S^1}, (x)) \cdot(W_I,W_{S^1}, (y))=(V_I\otimes W_I,V_{S^1}\otimes W_{S^1}, (x\otimes y))$.
    \end{itemize}
    \begin{proof}
        The axioms of $\mathcal{T}_{\text{KFA}}$ are directly verified for the sum and the product. This is a consequence of the fact that each axiom but one are in the form $D=D'$ where $D$ and $D'$ are connected diagrams. The remaining axiom can be checked by hand. 
    \end{proof}
\end{proposition}

\end{subsection}

\begin{subsection}{The category $\catname{2Cob}_{O-C}$}

The category $\catname{2Cob}_{O-C}$ is the category whose objects are disjoint unions of circles $S^1$ and unit closed intervals $I=[0,1]$, and whose morphisms are cobordisms between them, namely, compact oriented smooth surfaces with corners, with specified boundaries. We thus get a category where the cobordisms have boundaries separated into two kinds: the usual in- and out-boundaries, which are disjoint unions of circles and unit intervals, and some "internal" boundaries. A precise definition can be found in \cite{lauda-pfeiffer-1}, in which it is shown that this category can be understood as the free symmetric monoidal category generated by specific sets of objects $O$, morphisms $M$, and relations $R$ (cf. theorem 4.2 in \cite{lauda-pfeiffer-1}). We have $O=\{S^1, I\}$, and $M$ and $R$ are given below. For convenience, we draw cobordisms horizontally and read them from left to right. \\

$\quad$ \underline{Morphisms}:

\[
\begin{matrix}
    \raisebox{-0.5\height}{\includegraphics[scale=0.5]{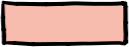}} &
    \raisebox{-0.5\height}{\includegraphics[scale=0.5]{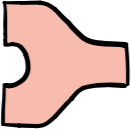}} &
    \raisebox{-0.5\height}{\includegraphics[scale=0.5]{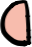}} & 
    \raisebox{-0.5\height}{\includegraphics[scale=0.5]{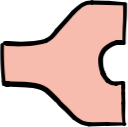}} & 
    \raisebox{-0.5\height}{\includegraphics[scale=0.5]{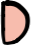}} 
    \\
    \text{(Identity }\mathds{1}_I\text{)} &
    \text{(Product }\nabla_I\text{)} &
    \text{(Unit } u_I\text{)} &
    \text{(Coproduct }\Delta_I\text{)} &
    \text{(Counit }\epsilon_I\text{)}
    \\
    & & & & \\
    \raisebox{-0.5\height}{\includegraphics[scale=0.5]{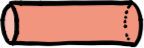}} &
    \raisebox{-0.5\height}{\includegraphics[scale=0.5]{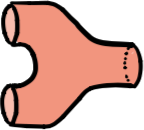}} &
    \raisebox{-0.5\height}{\includegraphics[scale=0.5]{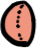}} & 
    \raisebox{-0.5\height}{\includegraphics[scale=0.5]{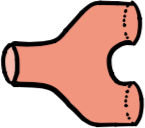}} & 
    \raisebox{-0.5\height}{\includegraphics[scale=0.5]{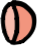}} 
    \\
    \text{(Identity }\mathds{1}_{S^1}\text{)} &
    \text{(Product }\nabla_{S^1}\text{)} &
    \text{(Unit } u_{S^1}\text{)} &
    \text{(Coproduct }\Delta_{S^1}\text{)} &
    \text{(Counit }\epsilon_{S^1}\text{)} 
\end{matrix}
\]
\[
\begin{matrix}
    \raisebox{-0.5\height}{\includegraphics[scale=0.5]{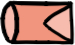}} & \vspace{0.5em} &
    \raisebox{-0.5\height}{\includegraphics[scale=0.5]{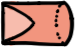}} 
    \\
    \text{(Zipper }\iota\text{)} & \qquad &
    \text{(Cozipper }\iota^{*}\text{)} 
\end{matrix} \vspace{0.5em}
\]

$\quad$ \underline{Relations}:

\begin{gather*}
    \begin{matrix} 
        \raisebox{-0.4\height}{\includegraphics[scale=0.5]{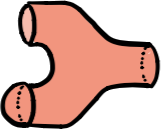}} & = & \raisebox{-0.4\height}{\includegraphics[scale=0.5]{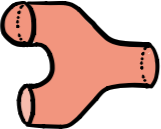}} & = & \raisebox{-0.3\height}{\includegraphics[scale=0.5]{figures/dessins/morphisms/id-S.png}} & = & \raisebox{-0.4\height}{\includegraphics[scale=0.5]{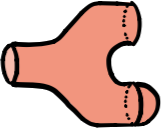}} & = & \raisebox{-0.4\height}{\includegraphics[scale=0.5]{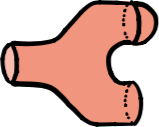}}
    \end{matrix}
    \\
    \begin{matrix} 
        \raisebox{-0.45\height}{\includegraphics[scale=0.5]{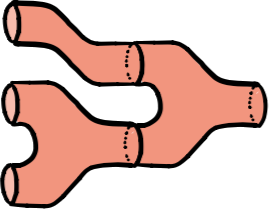}} & = & \raisebox{-0.45\height}{\includegraphics[scale=0.5]{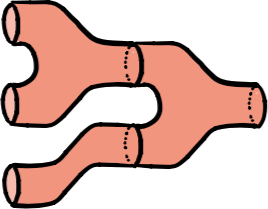}} & \quad & \raisebox{-0.45\height}{\includegraphics[scale=0.5]{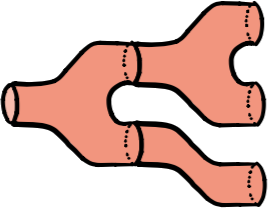}} & = & \raisebox{-0.45\height}{\includegraphics[scale=0.5]{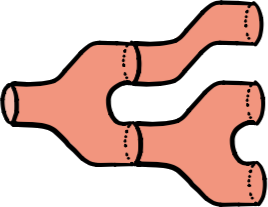}}
    \end{matrix}
    \\
    \begin{matrix} 
        \raisebox{-0.45\height}{\includegraphics[scale=0.5]{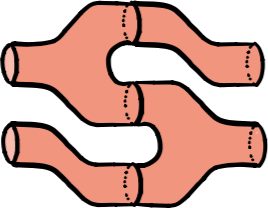}} & = & \raisebox{-0.45\height}{\includegraphics[scale=0.5]{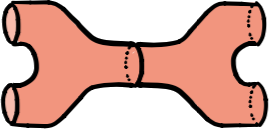}} & = & \raisebox{-0.45\height}{\includegraphics[scale=0.5]{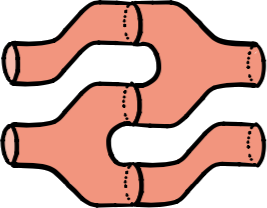}}
    \end{matrix}
    \\
    \begin{matrix}
        \begin{matrix} 
            \raisebox{-0.4\height}{\includegraphics[scale=0.5]{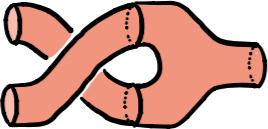}} & = & \raisebox{-0.4\height}{\includegraphics[scale=0.5]{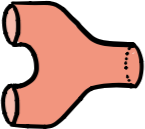}} & \quad & \raisebox{-0.4\height}{\includegraphics[scale=0.5]{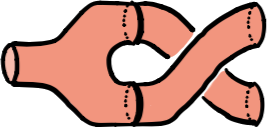}} & = & \raisebox{-0.4\height}{\includegraphics[scale=0.5]{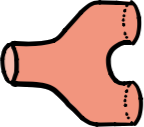}}
        \vspace{0.5em} \end{matrix}
        \\
        \text{(The closed sector forms a commutative Frobenius algebra)}
    \end{matrix}
\end{gather*}

\begin{gather*}
    \begin{matrix} 
        \raisebox{-0.4\height}{\includegraphics[scale=0.5]{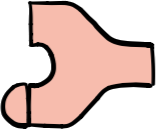}} & = & \raisebox{-0.4\height}{\includegraphics[scale=0.5]{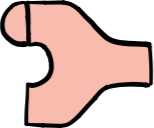}} & = & \raisebox{-0.3\height}{\includegraphics[scale=0.5]{figures/dessins/morphisms/id-I.png}} & = & \raisebox{-0.4\height}{\includegraphics[scale=0.5]{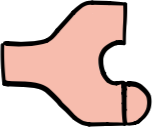}} & = & \raisebox{-0.4\height}{\includegraphics[scale=0.5]{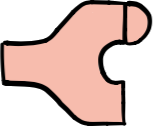}}
    \end{matrix}
    \\
    \begin{matrix} 
        \raisebox{-0.45\height}{\includegraphics[scale=0.5]{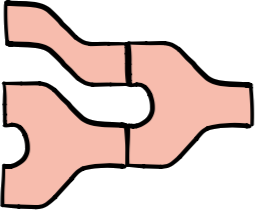}} & = & \raisebox{-0.45\height}{\includegraphics[scale=0.5]{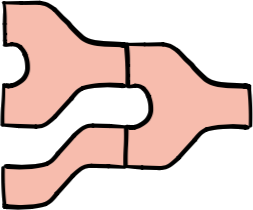}} & \quad & \raisebox{-0.45\height}{\includegraphics[scale=0.5]{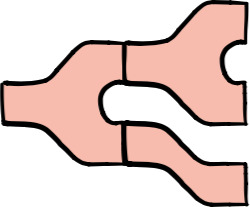}} & = & \raisebox{-0.45\height}{\includegraphics[scale=0.5]{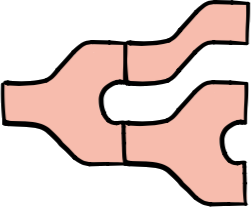}}
    \end{matrix}
    \\
    \begin{matrix} 
        \raisebox{-0.45\height}{\includegraphics[scale=0.5]{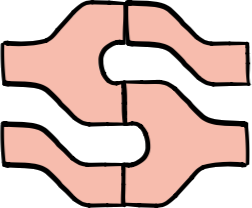}} & = & \raisebox{-0.45\height}{\includegraphics[scale=0.5]{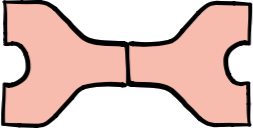}} & = & \raisebox{-0.45\height}{\includegraphics[scale=0.5]{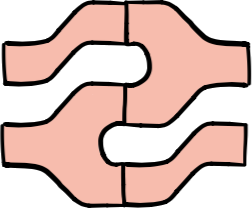}}
    \end{matrix}
    \\
    \begin{matrix}
        \begin{matrix} 
            \raisebox{-0.4\height}{\includegraphics[scale=0.5]{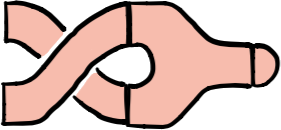}} & = & \raisebox{-0.4\height}{\includegraphics[scale=0.5]{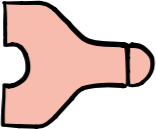}} & \quad & \raisebox{-0.4\height}{\includegraphics[scale=0.5]{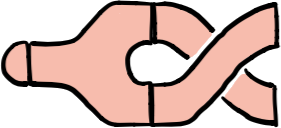}} & = & \raisebox{-0.4\height}{\includegraphics[scale=0.5]{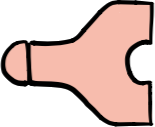}}
        \vspace{0.5em} \end{matrix}
        \\
        \text{(The open sector forms a symmetric Frobenius algebra)}
    \end{matrix}
\end{gather*}

\begin{gather*}
    \begin{matrix}
        \begin{matrix} 
            \raisebox{-0.3\height}{\includegraphics[scale=0.5]{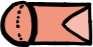}} & = & \raisebox{-0.3\height}{\includegraphics[scale=0.5]{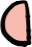}} & \quad \quad & \raisebox{-0.4\height}{\includegraphics[scale=0.5]{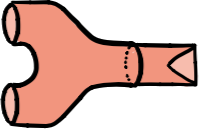}} & = & \raisebox{-0.4\height}{\includegraphics[scale=0.5]{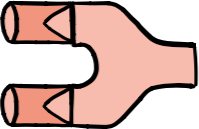}}
        \vspace{0.5em} \end{matrix} 
        \\
        \text{(} \iota \text{ is an algebra homomorphism)}
    \end{matrix}
\end{gather*}

\begin{gather*}
    \begin{matrix}
        \begin{matrix}
            \raisebox{-0.4\height}{\includegraphics[scale=0.5]{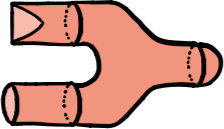}} & = & \raisebox{-0.4\height}{\includegraphics[scale=0.5]{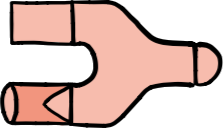}}
        \end{matrix} 
        & \qquad & 
        \begin{matrix}
            \raisebox{-0.4\height}{\includegraphics[scale=0.5]{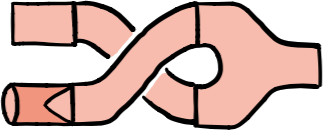}} & = & \raisebox{-0.4\height}{\includegraphics[scale=0.5]{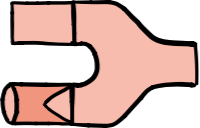}}
        \end{matrix} \vspace{0.5em}
        \\
        \text{(} \iota^* \text{ is dual to } \iota \text{)} & & \text{("Knowledge" about the center)}
    \end{matrix}
\end{gather*}

\begin{gather*}
    \begin{matrix}
        \begin{matrix}
            \raisebox{-0.4\height}{\includegraphics[scale=0.5]{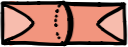}} & = & \raisebox{-0.4\height}{\includegraphics[scale=0.5]{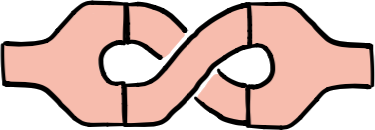}}
        \end{matrix} \vspace{0.5em}
        \\
        \text{(Cardy relation)}
    \end{matrix}
\end{gather*}
We have labelled the morphisms and relations in order to make apparent that $\catname{2Cob}_{O-C}$ is generated by a knowledgeable Frobenius algebra. In fact, in light of remark \ref{remark-idempotent-completion}, the idempotent completion $(\mathcal{C}_{\text{univ}}^{\mathcal{T}_{\text{KFA}}})^{\text{IC}}$ of $\mathcal{C}_{\text{univ}}^{\mathcal{T}_{\text{KFA}}}$ and the additive closure $(\catname{2Cob}_{O-C})^{\oplus}$ of $\catname{2Cob}_{O-C}$ are isomorphic, both being, \textit{inter alia}, free symmetric monoidal categories generated by isomorphic sets of objects, morphisms, and relations (cf. also Theorem 4.2 in \cite{lauda-pfeiffer-1} and its proof). \\

The important consequence is the ring isomorphism:
\[
    \mathds{K}[X]_{\text{aug}}^{\mathcal{T}_{\text{KFA}}} = \mathcal{C}_{\text{univ}}^{\mathcal{T}_{\text{KFA}}}(\textbf{1},\textbf{1})\cong (\catname{2Cob}_{O-C})^{\oplus}(\emptyset, \emptyset).
\]

\end{subsection}

\begin{subsection}{Characterisation of $\catname{2Cob}_{O-C}(\emptyset,\emptyset)$ and other properties}

It is shown in \cite{lauda-pfeiffer-1} that any connected cobordism can be reduced to a "normal form", characterised by the in- and out-boundaries, the genus $g$, and the number $w$ of "windows". Using the normal form, and defining the \textit{handle} and \textit{window} endomorphisms as $G=\nabla_{S^1}\circ \Delta_{S^1}$ and $W=\iota^* \circ \iota$, we can deduce that the the connected cobordisms from $\emptyset$ to $\emptyset$ are all on the form:
\begin{align*}
    \Sigma_{g,w} \quad & = \quad \epsilon_{S^1} \circ W^w \circ G^g \circ u_{S^1} \\
    & = \quad  \raisebox{-0.4\height}{\includegraphics[scale=0.5]{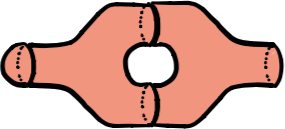}} \quad \cdots \quad \raisebox{-0.4\height}{\includegraphics[scale=0.5]{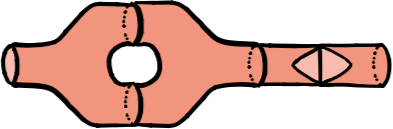}} \quad \cdots \quad \raisebox{-0.3\height}{\includegraphics[scale=0.5]{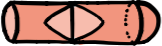}} 
\end{align*}
\vspace{-4.3em}
\begin{align*}
    \hphantom{\Sigma_{g,w} \quad} & \hphantom{= \quad  \raisebox{-0.4\height}{\includegraphics[scale=0.5]{figures/dessins/invariants/sigma-g-w-1.png}} \quad \cdots \quad \raisebox{-0.4\height}{\includegraphics[scale=0.5]{figures/dessins/invariants/sigma-g-w-2.png}} \quad \cdots \quad \raisebox{-0.3\height}{\includegraphics[scale=0.5]{figures/dessins/invariants/sigma-g-w-3.png}}} \\
    & \hphantom{xxxxxl} \underbrace{\hphantom{machinbiduleschouettemachinsl}}_{g \text{ handles}} \hphantom{x} \underbrace{\hphantom{fourbitrucmachinsl}}_{w \text{ windows}}
\end{align*}

for $g,w\in \mathds{N}$. In other words, $ \mathds{K}[X]_{\text{aug}}^{\mathcal{T}_{\text{KFA}}}\cong \mathds{K}[\{\Sigma_{g,w} : g,w\in \mathds{N} \}]\cong \mathds{K}[ \mathds{N}^2]$. \\

The cobordisms $\Sigma_{g,w}$ enjoy some characterisations in terms of traces of $W$, $G$, and $H$, which is another important endomorphism called the \textit{hole} or \textit{open window}:
\[
    H \quad := \quad \raisebox{-0.4\height}{\includegraphics[scale=0.5]{figures/dessins/morphisms/delta-I.png}}\hspace{-0.15em}\raisebox{-0.4\height}{\includegraphics[scale=0.5]{figures/dessins/morphisms/nabla-I.png}} \quad = \quad \nabla_I \circ \Delta_I.
\]

\begin{proposition}\label{properties-sigma-g-w}
For any $g,w\geqslant0$, we have the properties.
\begin{enumerate}
    \item $W\circ G = G\circ W$,
    \item $\Sigma_{g+1,w}=\text{Tr}_{S^1}(G^g \circ W^w):= \epsilon_{S^1} \circ \nabla_{S^1} \circ ((G^g \circ W^w) \otimes \mathds{1}_{S^1}) \circ \Delta_{S^1} \circ u_{S^1}$,
    \item $\Sigma_{0,w+2}=\text{Tr}_{I}(H^w):= \epsilon_I \circ \nabla_I \circ (H^w \otimes \mathds{1}_I) \circ \Delta_I \circ u_I $,
    \item $\Sigma_{0,1}=\epsilon_I \circ u_I$,
    \item $\Sigma_{0,0}=\epsilon_{S^1} \circ u_{S^1}$.
\end{enumerate}
Here, $\text{Tr}$ refers to taking the trace (i.e., closing the corresponding endomorphism of $I$ or $S^1$ by gluing its in-boundary to its out-boundary).
\end{proposition}
All these follow straightforwardly from the axioms and their consequences (cf. section 3.5 in \cite{lauda-pfeiffer-1}). \\

Note that any character $\chi:\mathds{K}[X]_{\text{aug}}^{\mathcal{T}_{\text{KFA}}}\rightarrow \mathds{K}$ produces a category $\mathcal{C}_{\chi}$ generated by a KFA object corresponding (as stated in table \ref{table:1}) to an open-closed TQFT $Z: \catname{2Cob}_{O-C}\rightarrow \calcatname{C}_{\chi}$. This TQFT $Z$ coincides with $\chi$ on all the $\Sigma_{g,w}$. In other words: in the broader context of Meir's framework, all TQFT scalar invariants of open-closed cobordisms are admissible. We will see however that an open-closed TQFT with target a good category (such as $\catname{Vec}_{\mathds{K}}$) is restricted to specific values on the $\Sigma_{g,w}$.

\end{subsection}

\end{section}
\begin{section}{Good knowledgeable Frobenius algebras}
\label{section-good-knfrob}

In the case of the theory $\mathcal{T}_{\text{KFA}}$, the set of characters $\text{Ch}_{\mathcal{T}_{\text{KFA}}}$ forms an algebra.
We show first that it is a subspace of $\text{Ch}_{\emptyset}$, and then we introduce a product.

\begin{proposition} 
    The subset $\text{Ch}_{\mathcal{T}_{\text{KFA}}} \subseteq \text{Ch}_{\emptyset}$ is closed under addition and multiplication by a scalar, as defined in \ref{definition-character}.
    \begin{proof}
        Let $\chi_1, \chi_2 \in \text{Ch}_{\emptyset}$. As shown in \cite{meir-interpolation,meir-universal}, $\chi_1+\chi_2$ is afforded by the direct sum $(V_1\oplus V_2,(x\oplus y))$ of two objects $(V_1,(x))$ and $(V_2,(y))$ of type $t_{\text{KFA}}$ in $\calcatname{C}_{\chi_1}\boxtimes\calcatname{C}_{\chi_2}$ affording respectively $\chi_1$ and $\chi_2$. But by proposition \ref{proposition-sum-product-KFAs}, if $V_1$ and $V_2$ are models for $\mathcal{T}_{\text{KFA}}$, so is $(V_1,(x))+(V_2,(y))$. Wherefore $\chi_1+\chi_2 \in \text{Ch}_{\mathcal{T}_{\text{KFA}}}$. \\
        Concerning multiplication by a scalar, let $\chi \in \text{Ch}_{\mathcal{T}_{\text{KFA}}}$. Let $I=\mathfrak{I}_{\mathcal{T}_{\text{FKA}}}(\textbf{1},\textbf{1})$ (the ideal such that $\mathds{K}[X]_{\text{aug}} / I = \mathds{K}[X]_{\text{aug}}^{\mathcal{T}_{\text{KFA}}}$), and let $p\in I$. We have that $(\lambda\chi)(p)$ is a polynomial in $\lambda$. Now, this polynomial must vanish for every $\lambda=n\in\mathds{N}$, because $n\chi = \sum_{i=1}^{n} \chi \in \text{Ch}_{\mathcal{T}_{\text{KFA}}}$ by closure under addition. Thus the polynomial must be $0$ for any $\lambda$, and consequently $(\lambda\chi)(I)=0$, showing that $\lambda\chi\in \text{Ch}_{\mathcal{T}_{\text{KFA}}}$.
    \end{proof}
\end{proposition}

As $\text{Ch}_{\emptyset}^{\mathcal{G}}$ and $\text{Ch}_{\mathcal{T}_{\text{KFA}}}$ are both subspaces of $\text{Ch}_{\emptyset}$, so is their intersection $\text{Ch}_{\mathcal{T}_{\text{KFA}}}^{\mathcal{G}}$. We can define a natural product on $\text{Ch}_{\mathcal{T}_{\text{KFA}}}$ turning it into an algebra. We recapitulate the algebra structure on $\text{Ch}_{\mathcal{T}_{\text{KFA}}}$ in the following definition.

\begin{definition} \label{definition-character-kfa} For two characters $\chi_1,\chi_2\in \text{Ch}_{\mathcal{T}_{\text{KFA}}}$, and any scalar $\lambda\in\mathds{K}$, we define the characters $\chi_1+\chi_2$ (addition), $\chi_1 \cdot \chi_2$ (product), and $\lambda\chi_1$ (multiplication by a scalar) by their actions on any closed connected diagram $\Sigma_{g,w}$ generating $\mathds{K}[X]_{\text{aug}}^{\mathcal{T}_{\text{KFA}}}$:
    \[  
    \left\{
    \begin{array}{lll}
        (\chi_1\cdot \chi_2)(\Sigma_{g,w}) & = & \chi_1(\Sigma_{g,w})\chi_2(\Sigma_{g,w}), \\
        (\chi_1+\chi_2)(\Sigma_{g,w}) & = & \chi_1(\Sigma_{g,w})+\chi_2(\Sigma_{g,w}), \\
        (\lambda\chi_1)(\Sigma_{g,w}) & = & \lambda(\chi_1(\Sigma_{g,w})). 
    \end{array} 
    \right.
    \]
\end{definition}

Note that with this structure, $\text{Ch}_{\mathcal{T}_{\text{KFA}}}$ is a subspace but not a subalgebra of $\text{Ch}_{\emptyset}$ (to distinguish the two products, we write the new one with an explicit "$\cdot$"). Furthermore, the subspace of good characters $\text{Ch}_{\mathcal{T}_{\text{KFA}}}^{\mathcal{G}}$ forms a subalgebra of $\text{Ch}_{\mathcal{T}_{\text{KFA}}}$:

\begin{proposition}
    Let $\chi_1,\chi_2\in \text{Ch}_{\mathcal{T}_{\text{KFA}}}^{\mathcal{G}}$. Then $\chi_1 \cdot \chi_2$ is good as well.
    \begin{proof}
        Let $V$ and $W$ be the KFAs models generating respectively $\calcatname{C}_{\chi_1}$ and $\calcatname{C}_{\chi_2}$. Both models can be seen as living in the same good category $\calcatname{C}_{\chi_1}\boxtimes\calcatname{C}_{\chi_2}$. Their product (cf. proposition \ref{proposition-sum-product-KFAs}), which is again a model for $\mathcal{T}_{\text{KFA}}$, affords $\chi_1 \cdot \chi_2$, as can be seen by direct evaluation on any closed connected diagram $\Sigma_{g,w}$. So $\chi_1 \cdot \chi_2$ is in $\text{Ch}_{\mathcal{T}_{\text{KFA}}}^{\mathcal{G}}$.
    \end{proof}
\end{proposition}

A character $\chi:\mathds{K}[X]_{\text{aug}}^{\mathcal{T}_{\text{KFA}}}\rightarrow\mathds{K}$, defines a sequence $\chi=(\chi_{g,w})_{g,w\in\mathds{N}}$ by setting $\chi_{g,w}:=\chi(\Sigma_{g,w})$, and vice versa. Similarly, there is a linear isomorphism between the space of characters $\text{Ch}_{\mathcal{T}_{\text{KFA}}}$ and $\mathds{K}[[X,Y]]$, the space of formal power series in two variables, given by: 

\[ \chi \mapsto f_{\chi}(X,Y) := \sum_{g,w\geqslant 0} \chi(\Sigma_{g,w})X^g Y^w = \sum_{g,w\geqslant 0} \chi_{g,w}X^g Y^w. \]

We call $f_{\chi}(X,Y)$ the \textit{generating function} of $\chi$.  Under some conditions, different generating functions will give characters producing equivalent categories:

\begin{proposition} \label{proposition-equivalent-generating-functions} (see also section 2.1 in \cite{kok}.)
    Let $\alpha\in\mathds{K}^*$ and $\chi, \chi'$ be two characters such that $f_{\chi'}(X,Y)=\frac{1}{\alpha^2}f_{\chi}(\alpha^2 X, \alpha Y)$. Then $\calcatname{C}_{\chi}\simeq\calcatname{C}_{\chi'}$ as tensor categories.
\end{proposition}
\begin{proof}
    Let $f$ be a connected cobordism with a total number $l\in\mathds{N}$ of \textit{in} and \textit{out} open boundaries. Define the scaling map $\phi$ which sends $f$ to $\alpha^{-e(f)+l/2} f$, where $e(f)$ is the Euler characteristic of $f$. This mapping extends to automorphisms of hom-spaces of $(\calcatname{C}^{\mathcal{T}_{\text{KFA}}}_{\text{univ}})^{\text{IC}}$ and is compatible with composition. Any character $\chi$ gives a pairing $pair_\chi$ on those hom-spaces (cf. section \ref{section-meir-tool}). Using, $\phi$, one can define another pairing given by precomposing $pair_{\chi}$ by $\phi\otimes \phi$. This new pairing corresponds to the pairing one would obtain from a character $\chi'$ defined by
    \[ \chi'(\Sigma_{g,w})=\chi(\Sigma_{g,w}) \alpha^{-e(\Sigma_{g,w})} \]
    where we recall that $e(\Sigma_{g,w})=2-2g-w$. As the isomorphism $\phi$ intertwines the two pairings $pair_{\chi}$ and $pair_{\chi'}$, the resulting quotient categories $\calcatname{C}_{\chi}$ and $\calcatname{C}_{\chi'}$ are equivalent as tensor categories. 
\end{proof}

Finding the subspace of good characters $\text{Ch}_{\mathcal{T}_{\text{KFA}}}^{\mathcal{G}}$ translates into finding the corresponding isomorphic subspace $\mathcal{G}\subseteq \mathds{K}[[X,Y]]$. To do so, we use the properties of proposition \ref{properties-sigma-g-w} and re-write $f_{\chi}$:
\begin{multline*}
f_{\chi}(X,Y)  = \chi(\Sigma_{0,0}) + \chi(\Sigma_{0,1})Y + Y^2\sum_{w\geqslant 0} \chi(\Sigma_{0,w+2})Y^w + X \sum_{g,w\geqslant 0} \chi(\Sigma_{g+1,w})X^g Y^w \\
 = \chi(\Sigma_{0,0}) + \chi(\Sigma_{0,1})Y + Y^2\sum_{w\geqslant 0} \text{Tr}_{I}(H^w)Y^w + X \sum_{g,w\geqslant 0} \text{Tr}_{S^1}(G^g \circ W^w)X^g Y^w,
\end{multline*}
where the traces are to be understood as traces in the tensor category $\calcatname{C}_{\chi}$. If this category is good, i.e., if $\chi$ is a good character, the propositions \ref{trace-of-endomorphism} and \ref{trace-of-endomorphisms} apply, and it follows that $\mathcal{G}$ necessarily sits inside the vector space spanned by the set

\[
    \left\{1, Y \right\} \cup \left\{ \frac{Y^2}{1-\nu Y} : \nu \in \mathds{K} \right\} \cup \left\{ \frac{X}{(1-\lambda X)(1-\mu Y)} : \lambda, \mu \in \mathds{K} \right\}, 
\]

\vspace{1em}
which is the same as the vector space spanned by 
\begin{multline*}
    \left\{1, X, Y, Y^2 \right\} \cup \left\{ \frac{1}{1-\lambda X}\frac{1}{1-\mu Y} : \lambda \in \mathds{K}^*, \mu \in \mathds{K} \right\} \\ 
    \cup \left\{ \frac{1}{1-\nu Y} : \nu \in \mathds{K}^* \right\} \cup \left\{ \frac{X}{1-\nu Y} : \nu \in \mathds{K}^* \right\}.
\end{multline*}

\begin{subsection}{Finding the generating functions for good characters}

\begin{proposition} \label{proposition-list-of-bad-characters}
Let $\nu\in\mathds{K}^*$. If $f_{\chi}(X,Y)=\frac{1}{1-\nu Y}$ or $f_{\chi}(X,Y)=\frac{X}{1-\nu Y}$, then $\chi$ is not a good character. 
\end{proposition}
\begin{proof}
    Note first that in both cases, the morphisms $G^2$ and $W-\nu \mathds{1}_{S^1}$ are $\chi$-negligible. Because of that, the morphism  $f\in\calcatname{C}_{\chi}(I^{\otimes3},I^{\otimes3})$ defined by 
    \[ f \quad = \quad \raisebox{-0.5\height}{\includegraphics[scale=0.5]{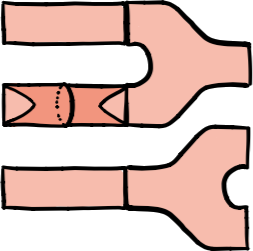}} \] 
    satisfies $f^4=0$ and $\text{Tr}f = \nu^4 \neq 0$, contradicting the last of the equivalent conditions of definition \ref{theorem-udi-good-character} for $\chi$ to be a good character. Wherefore the claim of the proposition holds.
\end{proof}

\begin{proposition} \label{proposition-list-of-good-characters}
    If $f_{\chi} \in \text{span}_{\mathds{K}}\left\{\left\{1, X, Y, Y^2 \right\} \cup \left\{ \frac{1}{1-\lambda X}\frac{1}{1-\mu Y} : \lambda \in \mathds{K}^*, \mu \in \mathds{K} \right\} \right\}$,
    then $\chi$ is good.
\end{proposition}
\begin{proof}
    It suffices to find models of KFAs in good categories affording characters $\chi$ whose associated power series $f_{\chi}$ span this vector space. We use the good examples given in section \ref{subsection-KFA} and gather the results in the following list. 
    \begin{itemize}
        \item \textbf{Case}: $f \in\text{span}_{\mathds{K}}\left\{\frac{1}{1-\lambda X}\frac{1}{1-\mu Y} : \lambda \in \mathds{K}^*, \mu \in \mathds{K}\right\}$. \\
            Let $\alpha\in\mathds{K}^*$. We consider the generalisation of $(F_{\alpha}(n),F_{\alpha^2},\iota, \iota^*)$ (example \ref{example-semisimple}), for $n=d\in\mathds{K}$, described in appendix \ref{section-annex-gl}. It lives in the good category $\catname{Rep}(GL_d)$. Direct calculation gives $f_{\chi}(X,Y)=\frac{1}{\lambda}\frac{1}{1-\lambda X}\frac{1}{1-\mu Y}$, where $\lambda=\alpha^{-2}$ and $\mu = d\alpha^{-1}$.
        \item \textbf{Case}: $f \in \text{span}_{\mathds{K}}\left\{1,X,Y,Y^2\right\}$. \\
            Let $p\in\mathds{N}, q\in\mathds{N}^*, \alpha\in\mathds{K}^*, \delta,\sigma\in\mathds{K}$, and $\beta=\frac{\alpha^2}{p+2}$.
            We consider the KFA $(A_{\alpha,\delta}(p), A_{\beta,\sigma}(q), \iota, \iota^*)$ of example \ref{example-non-semisimple} in the good category $\catname{Vec}_{\mathds{K}}$, This gives $f_{\chi}(X,Y)=\sigma + (q+2) X + \delta Y + (p+2) Y^2$.
    \end{itemize}
\end{proof}

Put together, the propositions \ref{proposition-list-of-bad-characters} and \ref{proposition-list-of-good-characters} give a proof of the claim in theorem \ref{theorem-main}, i.e., we have 
\[
    \mathcal{G}=\text{span}_{\mathds{K}}\left\{\left\{1, X, Y, Y^2 \right\} \cup \left\{ \frac{1}{1-\lambda X}\frac{1}{1-\mu Y} : \lambda \in \mathds{K}^*, \mu \in \mathds{K} \right\} \right\}.
\] \\

It is clear now that no KFA in the good category of finite-dimensional vector spaces can afford a generating function outside of $\mathcal{G}$, hence restricting the admissible open-closed TQFTs with target category $\catname{Vec}_{\mathds{K}}$.  More generally, we study in the next section the different kinds of good target categories we can use for open-closed TQFTs.

\end{subsection}

\end{section}

\begin{section}{The categories $\calcatname{C}_{\chi}$ when $\chi$ is good}
\label{section-rep-categories}

Let $\chi\in\text{Ch}_{\mathcal{T}_{\text{KFA}}}^{\mathcal{G}}$ be a good character with associated power series
\[
    f_{\chi}(X,Y)= \alpha_1+\alpha_X X+\alpha_Y Y+\alpha_{Y^2} Y^2 + \sum_{(\lambda,\mu)\in\mathcal{A}} \alpha_{\lambda,\mu}\frac{1}{1-\lambda X}\frac{1}{1-\mu Y},
\]
where $\mathcal{A}$ is a finite subset of $\mathds{K}^*\times\mathds{K}$, and $\alpha_1,\alpha_X,\alpha_Y ,\alpha_{Y^2} \in\mathds{K}, \alpha_{\lambda,\mu}\in\mathds{K}^*$, i.e., $\chi$ is the character satisfying: \\
\[
    \chi\left( \Sigma_{g,w}\right)=
    \sum_{(\lambda,\mu)\in\mathcal{A}} \alpha_{\lambda,\mu} \lambda^g \mu^w + 
        \left\{
        \begin{array}{ll}
            \alpha_1 & \text{ if } g=0, w=0 \\
            \alpha_X & \text{ if } g=1, w=0 \\
            \alpha_Y & \text{ if } g=0, w=1 \\
            \alpha_{Y^2} & \text{ if } g=0, w=2 \\
            0 & \text{ else}
        \end{array}
        \right.  .
\]

We are interested in the description of the category $\calcatname{C}_{\chi}$. We first show that it is equivalent to the Deligne product of two categories, splitting in some sense the semisimple and the non-semisimple parts of the KFA generating $\calcatname{C}_{\chi}$, and then describe each category as an interpolation of some categories of representations of groups.

\begin{subsection}{$\calcatname{C}_{\chi}\simeq \calcatname{C}_{\chi}^{\text{ss}}\boxtimes\calcatname{C}_{\chi}^{\text{non-ss}}$}

We claim the following.
\begin{proposition} \label{proposition-splitting-category} We have an equivalence of categories
    \[
        \calcatname{C}_{\chi}\simeq \calcatname{C}_{\chi}^{\text{ss}}\boxtimes\calcatname{C}_{\chi}^{\text{non-ss}}
    \]
    where 
    \begin{itemize}
        \item $\calcatname{C}_{\chi}^{\text{ss}}\simeq \underset{(\lambda,\mu)\in\mathcal{A}}{\boxtimes} \calcatname{C}_{\alpha_{\lambda,\mu}\chi_{(\lambda,\mu)}}$, with $f_{\chi_{(\lambda,\mu)}}(X,Y) = \frac{1}{1-\lambda X}\frac{1}{1-\mu Y}$,
        \item $\calcatname{C}_{\chi}^{\text{non-ss}}=\calcatname{C}_{\chi^{\text{non-ss}}}$, with $f_{\chi^{\text{non-ss}}}(X,Y)=\alpha_1+\alpha_X X+\alpha_Y Y+\alpha_{Y^2} Y^2$.
    \end{itemize}
\end{proposition}
To prove this statement, we will show that the generating KFA of $\mathcal{C}_{\chi}$ splits as a direct sum of sub-KFAs, one for each of the Deligne factors in the above decomposition. We start by proving some lemmata. \\

Define $\mathcal{A}_{X}$ (resp. $\mathcal{A}_{Y}$) to be the image of $\mathcal{A}$ under the projection of $\mathds{K}^*\times \mathds{K}$ on its first factor (resp. second factor). In other words, $\mathcal{A}_{X}$ is the set of distinct $\lambda$'s and $\mathcal{A}_{Y}$ the set of distinct $\mu$'s in $\mathcal{A}$.

\begin{lemma} \label{lemma-minimal-polynomial-G}
    Seen as an endomorphism in $\calcatname{C}_{\chi}$, the handle $G$ has minimal polynomial 
    \[ g(t)=t^k\prod_{\lambda\in\mathcal{A}_{X}} (t-\lambda) 
        \qquad \text{for some } k\in\{0,1,2\}.
    \]
    The value of $k$ depends on the vanishing of $\alpha_1, \alpha_X, \alpha_Y$, and  $\alpha_{Y^2}$. The precise value is irrelevant for our discussion.
\end{lemma}
\begin{proof} 
    We must show that $g(t)$ is the minimal polynomial such that $g(G)$, when seen as an endomorphism in $(\calcatname{C}_{\text{univ}}^{\mathcal{T}_{\text{KFA}}})^{\text{IC}}$, is $\chi$-negligible. Call $S^1$ the image of $\mathds{1}_{S^1}$ in $(\mathcal{C}_{\text{univ}}^{\mathcal{T}_{\text{KFA}}})^{\text{IC}}$. The endomorphisms of $S^1$ form a $\mathds{K}[X]_{\text{aug}}^{\mathcal{T}_{\text{KFA}}}$-module spanned by the sets $\{ \sigma_{g,w} : g,w \in \mathds{N} \}$ and $\{ \sigma_{x,y} \circ u_{S^1} \circ \epsilon_{S^1} \circ \sigma_{z,t} : x,y,z,t \in \mathds{N} \}$, where $\sigma_{n,m}:=G^n\circ W^m$. As $G$ commutes with $\sigma_{m,n}$, so does $g(G)$. It is thus sufficient to show 

    \begin{equation} \label{equation-1-lemma-minimal-polynomial-G}
        \forall p,q\in\mathds{N},~\chi\left(\epsilon_{S^1} \circ \sigma_{p,q} \circ g(G) \circ u_{S^1}\right)=0
    \end{equation}
    because 
    \[
        \begin{array}{rcl}
            \text{Tr}(\sigma_{g,w}\circ g(G)) & = & \epsilon_{S^1}\circ\sigma_{g+1,w}\circ g(G)\circ u_{S^1}, \\
            \text{Tr}(\sigma_{x,y} \circ u_{S^1} \circ \epsilon_{S^1} \circ \sigma_{z,t} \circ g(G)) & = &  \epsilon_{S^1}\circ\sigma_{x+z,y+t}\circ g(G)\circ u_{S^1}.
        \end{array}
    \]
    In order to prove eq. (\ref{equation-1-lemma-minimal-polynomial-G}), we rewrite: 
    \[
        g(G) = G^k
        \sum_{J\subseteq\mathcal{A}_X} G^{|\mathcal{A}_X|-|J|} (-1)^{|J|}  \prod_{x\in J} x
    \]
    such that, for all $p,q\in \mathds{N}$, 
    \begin{align*}
        \chi\left(\epsilon_{S^1} \circ \sigma_{p,q} \circ g(G) \circ u_{S^1}\right) & = \sum_{J\subseteq\mathcal{A}_X} \chi\left(\epsilon_{S^1} \circ \sigma_{p+k+|\mathcal{A}_X|-|J|, q} \circ u_{S^1}\right) (-1)^{|J|}  \prod_{x\in J} x \\
        & = \sum_{J\subseteq\mathcal{A}_X} \sum_{(\lambda,\mu)\in\mathcal{A}} \alpha_{\lambda,\mu} \lambda^{p+k+|\mathcal{A}_X|-|J|} \mu^q  (-1)^{|J|}  \prod_{x\in J} x \\
        & = \sum_{(\lambda,\mu)\in\mathcal{A}} \alpha_{\lambda,\mu} \lambda^{p+k} \mu^q
        \sum_{J\subseteq\mathcal{A}_X}  \lambda^{|\mathcal{A}_X|-|J|}   (-1)^{|J|}  \prod_{x\in J} x \\
        & = \sum_{(\lambda,\mu)\in\mathcal{A}} \alpha_{\lambda,\mu} \lambda^{p+k} \mu^q
        \prod_{x\in\mathcal{A}_X} (\lambda - x) = 0,
    \end{align*}
    where the second equality is valid for large enough $k$ (in general, $k=2$ is minimal, but the vanishing of some of the $\alpha_1, \alpha_X, \alpha_Y, \alpha_{Y^2}$ might allow for taking $k=1$ or even $k=0$). 
\end{proof}

We have orthogonal idempotent endomorphisms of $S^1$ in $\calcatname{C}_{\chi}$:
\[
\begin{matrix}
    e_{\lambda} & = & \displaystyle\frac{G^k}{\lambda^k}\prod_{\lambda \neq \lambda' \in \mathcal{A}_X} \frac{G-\lambda'}{\lambda-\lambda'}, & \lambda \in \mathcal{A}_X, & \qquad & 
    \Tilde{e} & = & \displaystyle \mathds{1}_{S^1}-\sum_{\lambda \in \mathcal{A}_X} e_{\lambda}
\end{matrix}.
\]

\begin{proposition} \label{proposition-decomposition-S-1}
    The direct sum decomposition $S^1= \text{Im}(\Tilde{e}) \bigoplus_{\lambda\in\mathcal{A}_X} \text{Im}(e_{\lambda})$ is a direct sum of commutative sub-Frobenius algebras.
\end{proposition}
\begin{proof} The proof relies on the fact that the handle $G$ can be freely moved in any cobordism of the closed sector (see proposition 3.16 in \cite{lauda-pfeiffer-1}) -- and so do all the aforementioned idempotents. As a result, the commutative Frobenius algebra structure acts separately on each summand. For instance, for $\lambda_1, \lambda_2, \lambda_3 \in \mathcal{A}_X$, we have $e_{\lambda_1} \circ \nabla_{S^1} \circ (e_{\lambda_2} \otimes e_{\lambda_3}) = e_{\lambda_1} \circ e_{\lambda_2} \circ e_{\lambda_3} \circ \nabla_{S^1}$ which is $0$ if the three $\lambda_1, \lambda_2, \lambda_3$ are different (with similar results for the rest of the structure).
\end{proof}

The window $W$ commutes with $G$, so it preserves this decomposition. 
\begin{lemma} 
    Let $\lambda\in\mathcal{A}_X$. On $\text{Im}(e_{\lambda})$, the restriction of $W$ has minimal polynomial 
    \[ w_{\lambda}(t)=\prod_{\substack{\mu\in\mathcal{A}_Y \\ (\lambda,\mu)\in \mathcal{A}}} (t-\mu) .\]
\end{lemma}
\begin{proof}
    We must show that 
    \[ w_{\lambda}(W|_{\text{Im}(e_{\lambda})}) = \prod_{\substack{\mu\in\mathcal{A}_Y \\ (\lambda,\mu)\in \mathcal{A}}} (e_{\lambda} \circ W \circ e_{\lambda} - \mu e_{\lambda}) = e_{\lambda} \prod_{\substack{\mu\in\mathcal{A}_Y \\ (\lambda,\mu)\in \mathcal{A}}} (W - \mu)  \]
    is $\chi$-negligible. This is done following \textit{mutatis mutandis} the proof of lemma \ref{lemma-minimal-polynomial-G}, noting that
    \[
        \chi(\epsilon_{S^1} \circ e_{\lambda} \circ \sigma_{m,n} \circ u_{S^1}) = \lambda^m \sum_{\substack{\mu\in\mathcal{A}_Y \\ (\lambda,\mu)\in \mathcal{A}}} \alpha_{\lambda,\mu}  \mu^n.
    \] 
\end{proof}

We have orthogonal idempotent endomorphisms of $S^1$ in $\calcatname{C}_{\chi}$:
\[
    e_{\lambda,\mu} = e_{\lambda} \displaystyle\prod_{\substack{\mu \neq \mu' \in \mathcal{A}_Y \\ (\lambda,\mu') \in \mathcal{A}}} \frac{W-\mu'}{\mu-\mu'}, \qquad (\lambda,\mu) \in \mathcal{A}.
\]

We can decompose $S^1$ further, and by a similar argument as for proposition \ref{proposition-decomposition-S-1}, we obtain 
\begin{proposition} \label{proposition-decomposition-S-2}
    The direct sum decomposition $S^1= \text{Im}(\Tilde{e}) \bigoplus_{(\lambda,\mu)\in\mathcal{A}} \text{Im}(e_{\lambda,\mu})$ is a direct sum of commutative sub-Frobenius algebras.
\end{proposition}

Let us now turn to the decomposition of the symmetric Frobenius algebra part. 

\begin{lemma} \label{lemma-minimal-polynomial-H}
    Seen as an endomorphism in $\calcatname{C}_{\chi}$, the hole $H$ has minimal polynomial 
    \[ h(t)=t^{k'}\prod_{\mu\in\mathcal{A}_{Y}} (t-\mu) 
        \qquad \text{for some } k'\in\{0,1,2\}.
    \]
    The value of $k'$ depends on whether $0\in\mathcal{A}_Y$, and on the vanishing of $\alpha_1, \alpha_X, \alpha_Y$, and  $\alpha_{Y^2}$. The precise value is irrelevant for our discussion.
\end{lemma}
\begin{proof} 
    We must show that $h(t)$ is the minimal polynomial such that $h(H)$, when seen as an endomorphism in $(\calcatname{C}_{\text{univ}}^{\mathcal{T}_{\text{KFA}}})^{\text{IC}}$, is $\chi$-negligible. Call $I$ the image of $\mathds{1}_{I}$ in $(\mathcal{C}_{\text{univ}}^{\mathcal{T}_{\text{KFA}}})^{\text{IC}}$. The endomorphisms of $I$ form a $\mathds{K}[X]_{\text{aug}}^{\mathcal{T}_{\text{KFA}}}$-module spanned by the sets $\{\mathds{1}_I\}$, $\{ \iota \circ \sigma_{g,w} \circ \iota^* : g,w \in \mathds{N} \}$ and $\{ \iota \circ \sigma_{x,y} \circ u_{S^1} \circ \epsilon_{S^1} \circ \sigma_{z,t} \circ \iota^* : x,y,z,t \in \mathds{N} \}$, where $\sigma_{n,m}:=G^n\circ W^m$. It is sufficient to show 

    \begin{equation} \label{equation-1-lemma-minimal-polynomial-H}
        \forall p,q\in\mathds{N},~\chi\left(\epsilon_{S^1} \circ \sigma_{p,q+1} \circ h(W) \circ u_{S^1}\right)=0
    \end{equation}
    because
    \[
        \begin{array}{rcl}
            \iota^*\circ H^n \circ \iota & = & W^{n+1} \qquad \forall n\in\mathds{N}, \\
            \text{Tr}(h(H)) & = & \epsilon_{S^1} \circ \sigma_{0,2} \circ h(W) \circ u_{S^1}, \\
            \text{Tr}(\iota \circ \sigma_{g,w} \circ \iota^* \circ h(H)) & = & \epsilon_{S^1} \circ \sigma_{g+1,w+1} \circ h(W) \circ u_{S^1}, \\
            \text{Tr}(\iota \circ \sigma_{x,y} \circ u_{S^1} \circ \epsilon_{S^1} \circ \sigma_{z,t} \circ \iota^* \circ h(H)) & = &  \epsilon_{S^1} \circ \sigma_{x+z,y+t+1} \circ h(W) \circ u_{S^1}.
        \end{array}
    \]
    In order to prove eq. (\ref{equation-1-lemma-minimal-polynomial-H}), we rewrite: 
    \[
        h(W) = W^{k'}
        \sum_{J\subseteq\mathcal{A}_Y} W^{|\mathcal{A}_Y|-|J|} (-1)^{|J|}  \prod_{x\in J} x
    \]
    such that, for all $p,q\in \mathds{N}$, 
    \begin{align*}
        \chi\left(\epsilon_{S^1} \circ \sigma_{p,q+1} \circ h(W) \circ u_{S^1}\right) & = \sum_{J\subseteq\mathcal{A}_Y} \chi\left(\epsilon_{S^1} \circ \sigma_{p, q+1+k'+|\mathcal{A}_Y|-|J|} \circ u_{S^1}\right) (-1)^{|J|}  \prod_{x\in J} x \\
        & = \sum_{J\subseteq\mathcal{A}_Y} \sum_{(\lambda,\mu)\in\mathcal{A}} \alpha_{\lambda,\mu} \lambda^{p} \mu^{q+1+k'+|\mathcal{A}_Y|-|J|}  (-1)^{|J|}  \prod_{x\in J} x \\
        & = \sum_{(\lambda,\mu)\in\mathcal{A}} \alpha_{\lambda,\mu} \lambda^{p} \mu^{q+k'+1}
        \sum_{J\subseteq\mathcal{A}_Y}  \mu^{|\mathcal{A}_Y|-|J|}   (-1)^{|J|}  \prod_{x\in J} x \\
        & = \sum_{(\lambda,\mu)\in\mathcal{A}} \alpha_{\lambda,\mu} \lambda^{p} \mu^{q+k'+1}
        \prod_{x\in\mathcal{A}_Y} (\mu - x) = 0,
    \end{align*}
    where the second equality is valid for large enough $k'$ (in general, $k'=2$ is minimal, but it can be lower depending on whether $0\in\mathcal{A}_Y$, and on the vanishing of some of the $\alpha_1, \alpha_X, \alpha_Y, \alpha_{Y^2}$). 
\end{proof}

We have orthogonal idempotent endomorphisms of $I$ in $\calcatname{C}_{\chi}$:
\[
\begin{matrix}
    a_{\mu} & = & \displaystyle\frac{H^k}{\mu^k}\prod_{\mu \neq \mu' \in \mathcal{A}_Y} \frac{H-\mu'}{\mu-\mu'}, & \mu \in \mathcal{A}_Y \backslash \{0\}, & \quad & 
    \Tilde{a} & = & \displaystyle \mathds{1}_{I}-\sum_{\mu \in \mathcal{A}_Y \backslash \{0\}} a_{\mu}
\end{matrix}.
\]

\begin{proposition} \label{proposition-decomposition-I-1}
    The direct sum decomposition $I= \text{Im}(\Tilde{a}) \bigoplus_{\mu\in\mathcal{A}_Y \backslash \{0\}} \text{Im}(a_{\mu})$ is a direct sum of symmetric sub-Frobenius algebras.
\end{proposition}
\begin{proof} The proof relies on the fact that the hole $H$ can be freely moved in any cobordism of the open sector (see proposition 3.15 in \cite{lauda-pfeiffer-1}) -- and so do all the aforementioned idempotents. As a result, the symmetric Frobenius algebra structure acts separately on each summand. For instance, for $\mu_1, \mu_2, \mu_3 \in \mathcal{A}_Y \backslash \{0\}$, we have $a_{\mu_1} \circ \nabla_{I} \circ (a_{\mu_2} \otimes a_{\mu_3}) = a_{\mu_1} \circ a_{\mu_2} \circ a_{\mu_3} \circ \nabla_{I}$ which is $0$ if the three $\mu_1, \mu_2, \mu_3$ are different (with similar results for the rest of the structure).
\end{proof}

Note that $W$ and $H$ have the same (generalized) eigenvalues. The two relations $\iota \circ W = H \circ \iota$ and $W \circ \iota^* = \iota^* \circ H$ imply that the zipper and cozipper "preserve" the eigenvalues, i.e., we can write them as sums of morphisms: $\iota=\Tilde{\iota}\bigoplus_{\mu\in\mathcal{A}_Y \backslash \{0\}} \iota_{\mu}$ and $\iota^*=\Tilde{\iota}^*\bigoplus_{\mu\in\mathcal{A}_Y \backslash \{0\}} \iota_{\mu}^*$, where
\begin{align*}
    \displaystyle \Tilde{\iota}: \text{Im}(\Tilde{e}) & \leftrightarrows \text{Im}(\Tilde{a}) : \Tilde{\iota}^* \\
    \displaystyle \iota_{\mu}: \bigoplus_{\substack{\lambda\in\mathcal{A}_X, \\ (\lambda,\mu)\in\mathcal{A}}}\text{Im}(e_{\lambda,\mu}) & \leftrightarrows \text{Im}(a_{\mu}) : \iota^*_{\mu}.
\end{align*}
    
Define the piecewise endomorphism $G':I\rightarrow I$ by: 
\[
    \left\{
    \begin{array}{lcl}
        G'|_{\text{Im}(\Tilde{a})} & = & 0 \\
        G'|_{\text{Im}(a_{\mu})} & = & \frac{1}{\mu} (\iota\circ G\circ \iota^*)|_{\text{Im}(a_{\mu})}
    \end{array}
    \right.
\]
which is well defined because $G$ preserves the full decomposition of $S^1$ stated in proposition \ref{proposition-decomposition-S-2}. Note also that $H$ and $G'$ commute.

\begin{lemma} 
    Let $\mu\in\mathcal{A}_Y\backslash\{0\}$. On $\text{Im}(a_{\mu})$, the restriction of $G'$ has minimal polynomial 
    \[ g'_{\mu}(t)=\prod_{\substack{\lambda\in\mathcal{A}_X \\ (\lambda,\mu)\in \mathcal{A}}} (t-\lambda) .\]
\end{lemma}
\begin{proof}
    We must show that 
    \[ g'_{\mu}(G'|_{\text{Im}(a_{\mu})}) = \prod_{\substack{\lambda\in\mathcal{A}_X \\ (\lambda,\mu)\in \mathcal{A}}} (a_{\mu} \circ \frac{1}{\mu} (\iota\circ G\circ \iota^*) \circ a_{\mu} - \lambda a_{\mu}) = a_{\mu} \prod_{\substack{\lambda\in\mathcal{A}_X \\ (\lambda,\mu)\in \mathcal{A}}} (\frac{1}{\mu} (\iota\circ G\circ \iota^*) - \lambda)  \]
    is $\chi$-negligible. This is done following \textit{mutatis mutandis} the proof of lemma \ref{lemma-minimal-polynomial-H}, noting that
    \[
        \chi(\epsilon_{S^1} \circ a_{\mu} \circ \sigma_{m,n} \circ u_{S^1}) = \mu^n \sum_{\substack{\lambda\in\mathcal{A}_X \\ (\lambda,\mu)\in \mathcal{A}}} \alpha_{\lambda,\mu}  \lambda^m.
    \] 
\end{proof}

We have orthogonal idempotent endomorphisms of $I$ in $\calcatname{C}_{\chi}$:
\[
    a_{\lambda,\mu} = a_{\mu} \displaystyle\prod_{\substack{\lambda' \in \mathcal{A}_X \\ (\lambda',\mu) \in \mathcal{A}}} \frac{G'-\lambda'}{\lambda-\lambda'}, \qquad (\lambda,\mu) \in \mathcal{A}, \mu\neq0.
\]

We can decompose $I$ further, and by a similar argument as for proposition \ref{proposition-decomposition-I-1}, we obtain 
\begin{proposition} \label{proposition-decomposition-I-2}
    The direct sum decomposition $I= \text{Im}(\Tilde{a}) \bigoplus_{\substack{(\lambda,\mu)\in\mathcal{A}, \mu\neq0}} \text{Im}(a_{\lambda,\mu})$ is a direct sum of symmetric sub-Frobenius algebras.
\end{proposition}

We are now ready to prove proposition \ref{proposition-splitting-category}.
\begin{proof}[Proof of proposition \ref{proposition-splitting-category}]
    In $\calcatname{C}_{\chi}$, the defining KFA, $I\oplus S^1$ splits as a direct sum of sub-KFAs,
    \[
        I\oplus S^1 \cong \text{Im}(\Tilde{a}) \oplus \text{Im}(\Tilde{e})
        \bigoplus_{\substack{(\lambda,\mu)\in\mathcal{A} \\ \mu \neq 0}}
        \text{Im}(a_{\lambda,\mu}) \oplus \text{Im}(e_{\lambda,\mu})
        \bigoplus_{(\lambda,0)\in\mathcal{A}}
        0 \oplus \text{Im}(e_{\lambda,0}),
    \]
    where, for each couple, the restrictions of $\iota$ and $\iota^*$ indeed act as well-defined zipper and cozipper. It follows from the values of the $\Sigma_{g,w}$, $g,w\in\mathds{N}$, on each sub-KFA, that: 
    \begin{itemize}
        \item $\text{Im}(a_{\lambda, \mu})\oplus \text{Im}(e_{\lambda, \mu})$, with $(\lambda,\mu)\in\mathcal{A}, \mu\neq0$, affords the character $\alpha_{\lambda, \mu}\chi(\lambda,\mu)$;
        \item $0\oplus \text{Im}(e_{\lambda, 0})$, with $(\lambda,0)\in\mathcal{A}$, affords the character $\alpha_{\lambda, 0}\chi(\lambda,0)$;
        \item $\text{Im}(\Tilde{a})\oplus \text{Im}(\Tilde{e})$ affords the character $\chi^{\text{non-ss}}$.
    \end{itemize}
    Moreover, this decomposition makes apparent the equivalence of categories stated in the proposition.
\end{proof}

We turn now to the description of $\calcatname{C}_{\chi^{\text{non-ss}}}$ and $\calcatname{C}_{\alpha_{\lambda,\mu}\chi(\lambda,\mu)}$ in terms of more familiar categories of representation.

\end{subsection}

\begin{subsection}{The category $\calcatname{C}_{\chi}^{\text{ss}}$}

Let $t,\lambda\in \mathds{K}^*$ and $\mu\in\mathds{K}$. We are interested in the case of the character $\frac{t}{\lambda}\chi(\lambda,\mu)$ which has a generating function on the form
\[
    f_{\frac{t}{\lambda}\chi(\lambda,\mu)}(X,Y) = \frac{t}{\lambda}\frac{1}{1-\lambda X}\frac{1}{1-\mu Y}.
\]
Let $a\in\mathds{K}^*$ be such that $a^2=\lambda$ and define $d:=\frac{\mu}{a}$. By taking $\alpha=\frac{1}{a}$ in proposition \ref{proposition-equivalent-generating-functions}, we see that the generating function
\[
    f_{t\chi(1,d)}(X,Y) = \frac{t}{1-X}\frac{1}{1- d Y}
\]
yields a category $\calcatname{C}_{t\chi(1,d)}$ equivalent to the category $\calcatname{C}_{t\chi(\lambda,\mu)}$. (By taking $\alpha=\frac{1}{-a}$, we obtain  $\calcatname{C}_{t\chi(1,d)}\simeq \calcatname{C}_{t\chi(1,-d)}$.) \\

In the special case $t=1$, we have a direct means to characterise more concretely the category $\calcatname{C}_{\chi(1,d)}$ by using the proposition \ref{proposition-rep-automorphism-group}, as we will see below. For $t\neq 1$, the situation is more complex and is discussed in section \ref{section-rep-categories-case-tneq1}.

\begin{subsubsection}{Case $t=1$, $d=n\in\mathds{N}$}

When $d=n\in\mathds{N}$ is a natural number and $t=1$, we have seen (cf. proposition \ref{proposition-list-of-good-characters}) that the KFA $B_n:=(F_1(n), F_1,\iota, \iota^*)$ (cf. example \ref{example-semisimple}) affords the character $\chi(1,n)$. 

\begin{proposition}
    The KFA $B_n$ has closed $GL_{n^2+1}$-orbit. 
\end{proposition}
\begin{proof}
    We recall first that the closure of an orbit contains a unique closed orbit (see, e.g., Proposition 2.7. in \cite{meir-interpolation}). Call $B$ the KFA with closed orbit in the closure of the orbit of $B_n$. The KFA $B_n=F_1(n)\oplus F_1$ is an associative unital semisimple algebra (with product $\nabla = \nabla_{F_1(n)}\oplus \nabla_{F_1}$) made of the direct sum of two simple parts. This semisimplicity is equivalent to the non-degeneracy of the following trace pairing: 
    \begin{align*}
        \langle - , - \rangle :  B_n\otimes B_n  & \rightarrow \mathds{K}   \\ 
        a\otimes b & \mapsto \langle a, b \rangle := \text{Tr}(L_a\circ L_b)
    \end{align*}
    where $\text{Tr}$ is the usual trace of endomorphism, and for $a\in B_n$, $L_a$ is the endomorphism 
    \[
        \begin{array}{rcccl}
            L_a & : & B_n  & \rightarrow &B_n \\ 
            & & x & \mapsto & \nabla(a\otimes x). 
        \end{array}
    \]
    This pairing can be written in terms of the structure morphisms as the evaluation in $\catname{Vec}_{\mathds{K}}$ of the following diagram
    \begin{align*}
        \begin{tikzpicture}
	\begin{pgfonlayer}{nodelayer}
		\node [style=box] (0) at (0, -0.75) {$\nabla$};
		\node [style=none] (1) at (0, -0.5) {};
		\node [style=none] (2) at (0.25, -1) {};
		\node [style=none] (3) at (-0.25, -1) {};
		\node [style=box] (4) at (-0.75, 0.75) {$\nabla$};
		\node [style=none] (5) at (-0.75, 1) {};
		\node [style=none] (6) at (-0.5, 0.5) {};
		\node [style=none] (7) at (-1, 0.5) {};
		\node [style=none] (8) at (-0.25, -2) {};
		\node [style=none] (9) at (-1, -2) {};
		\node [style=none] (10) at (-0.75, 1.25) {};
		\node [style=none] (11) at (0.25, 1.5) {};
		\node [style=none] (12) at (1, -1.5) {};
		\node [style=none] (13) at (0.25, -1.5) {};
		\node [style=none] (14) at (-0.5, 1.75) {};
	\end{pgfonlayer}
	\begin{pgfonlayer}{edgelayer}
		\draw [in=-90, out=90, looseness=1.50] (1.center) to (6.center);
		\draw (3.center) to (8.center);
		\draw (9.center) to (7.center);
		\draw (5.center) to (10.center);
		\draw (2.center) to (13.center);
		\draw [in=-45, out=75, looseness=0.75] (12.center) to (11.center);
		\draw [in=-90, out=-105, looseness=1.50] (12.center) to (13.center);
		\draw [in=-150, out=90] (10.center) to (14.center);
		\draw [in=135, out=30] (14.center) to (11.center);
	\end{pgfonlayer}
\end{tikzpicture}
 \quad = \quad \begin{tikzpicture}
	\begin{pgfonlayer}{nodelayer}
		\node [style=box] (0) at (0, -0.75) {$\nabla_I$};
		\node [style=none] (1) at (0, -0.5) {};
		\node [style=none] (2) at (0.25, -1) {};
		\node [style=none] (3) at (-0.25, -1) {};
		\node [style=box] (4) at (-0.75, 0.75) {$\nabla_I$};
		\node [style=none] (5) at (-0.75, 1) {};
		\node [style=none] (6) at (-0.5, 0.5) {};
		\node [style=none] (7) at (-1, 0.5) {};
		\node [style=none] (8) at (-0.25, -2) {};
		\node [style=none] (9) at (-1, -2) {};
		\node [style=none] (10) at (-0.75, 1.25) {};
		\node [style=none] (11) at (0.25, 1.5) {};
		\node [style=none] (12) at (1, -1.5) {};
		\node [style=none] (13) at (0.25, -1.5) {};
		\node [style=none] (14) at (-0.5, 1.75) {};
	\end{pgfonlayer}
	\begin{pgfonlayer}{edgelayer}
		\draw [in=-90, out=90, looseness=1.50] (1.center) to (6.center);
		\draw (3.center) to (8.center);
		\draw (9.center) to (7.center);
		\draw (5.center) to (10.center);
		\draw (2.center) to (13.center);
		\draw [in=-45, out=75, looseness=0.75] (12.center) to (11.center);
		\draw [in=-90, out=-105, looseness=1.50] (12.center) to (13.center);
		\draw [in=-150, out=90] (10.center) to (14.center);
		\draw [in=135, out=30] (14.center) to (11.center);
	\end{pgfonlayer}
\end{tikzpicture}
 + \begin{tikzpicture}
	\begin{pgfonlayer}{nodelayer}
		\node [style=box] (0) at (0.25, -0.75) {$\nabla_{S^1}$};
		\node [style=none] (1) at (0.25, -0.5) {};
		\node [style=none] (2) at (0.5, -1) {};
		\node [style=none] (3) at (0, -1) {};
		\node [style=box] (4) at (-0.75, 0.75) {$\nabla_{S^1}$};
		\node [style=none] (5) at (-0.75, 1) {};
		\node [style=none] (6) at (-0.5, 0.5) {};
		\node [style=none] (7) at (-1, 0.5) {};
		\node [style=none] (8) at (0, -2) {};
		\node [style=none] (9) at (-1, -2) {};
		\node [style=none] (10) at (-0.75, 1.25) {};
		\node [style=none] (11) at (0.75, 1.5) {};
		\node [style=none] (12) at (1.5, -1.5) {};
		\node [style=none] (13) at (0.5, -1.5) {};
		\node [style=none] (14) at (-0.5, 1.75) {};
	\end{pgfonlayer}
	\begin{pgfonlayer}{edgelayer}
		\draw [in=-90, out=90, looseness=1.50] (1.center) to (6.center);
		\draw (3.center) to (8.center);
		\draw (9.center) to (7.center);
		\draw (5.center) to (10.center);
		\draw (2.center) to (13.center);
		\draw [in=-45, out=75, looseness=0.75] (12.center) to (11.center);
		\draw [in=-90, out=-105, looseness=1.50] (12.center) to (13.center);
		\draw [in=-150, out=90] (10.center) to (14.center);
		\draw [in=135, out=30] (14.center) to (11.center);
	\end{pgfonlayer}
\end{tikzpicture}
 \quad & = \quad n \quad \begin{tikzpicture}
	\begin{pgfonlayer}{nodelayer}
		\node [style=box] (0) at (0, -0.5) {$\nabla_{I}$};
		\node [style=box] (1) at (0, 1) {$\epsilon_{I}$};
		\node [style=none] (2) at (-0.25, -0.75) {};
		\node [style=none] (3) at (0.25, -0.75) {};
		\node [style=none] (4) at (0, -0.25) {};
		\node [style=none] (5) at (0, 0.75) {};
		\node [style=none] (6) at (-0.25, -2) {};
		\node [style=none] (7) at (0.25, -2) {};
	\end{pgfonlayer}
	\begin{pgfonlayer}{edgelayer}
		\draw (4.center) to (5.center);
		\draw (2.center) to (6.center);
		\draw (3.center) to (7.center);
	\end{pgfonlayer}
\end{tikzpicture}
 \quad  +  \quad \begin{tikzpicture}
	\begin{pgfonlayer}{nodelayer}
		\node [style=box] (0) at (0, -0.5) {$\nabla_{S^1}$};
		\node [style=box] (1) at (0, 1) {$\epsilon_{S^1}$};
		\node [style=none] (2) at (-0.25, -0.75) {};
		\node [style=none] (3) at (0.25, -0.75) {};
		\node [style=none] (4) at (0, -0.25) {};
		\node [style=none] (5) at (0, 0.75) {};
		\node [style=none] (6) at (-0.25, -2) {};
		\node [style=none] (7) at (0.25, -2) {};
	\end{pgfonlayer}
	\begin{pgfonlayer}{edgelayer}
		\draw (4.center) to (5.center);
		\draw (2.center) to (6.center);
		\draw (3.center) to (7.center);
	\end{pgfonlayer}
\end{tikzpicture}
 \\
         & = \quad n (\epsilon_I \circ \nabla_I) + (\epsilon_{S^1} \circ \nabla_{S^1})
    \end{align*}
    where we have used the axioms of a Frobenius algebra and the fact that for the KFA $B_n$, we have that $H= n \mathds{1}_{F_1(n)}$ and $G= \mathds{1}_{F_1}$.
    Non-degeneracy means that there is a copairing $\gamma: \mathds{K}\rightarrow B_n \otimes B_n$, which is
    \[
        \gamma = \left\{ 
        \begin{matrix} \displaystyle
            \frac{1}{n} (\Delta_I \circ u_I) + (\Delta_{S^1} \circ u_{S^1}) & \text{ for } n\neq0 \\
            \Delta_{S^1} \circ u_{S^1} & \text{ for } n=0
        \end{matrix}
        \right.,
    \]
    and that with the pairing they satisfy the non-degeneracy conditions
    \[
        (\langle - ,- \rangle \otimes \mathds{1}_{B_n}) \circ (\mathds{1}_{B_n} \otimes \gamma) = \mathds{1}_{B_n} = (\mathds{1}_{B_n} \otimes \langle - ,- \rangle) \circ (\gamma \otimes \mathds{1}_{B_n}).
    \]
    We see that the semisimplicity condition can be expressed in terms of the KFA structure, i.e., as diagrams. This means that semisimplicity is preserved when we go to the closure of the orbit (indeed, the character afforded by $B_n$ and the ones afforded by the KFAs in the closure of the orbit of $B_n$ must coincide): all KFAs in the closure of the orbit of $B_n$ are semisimple associative unital algebras. By the Wedderburn-Artin theorem, any semisimple algebra is a direct sum of matrix algebras. Following proposition \ref{proposition-relation-symmetric-FA}, any KFA structure on a semisimple algebra is a direct sum of KFAs of type $F_{\alpha}(n)$. Combining these facts, $B$ must be isomorphic to $B_n$. Wherefore $B_n$ has closed orbit. 
\end{proof}

\begin{proposition}
    The automorphism group of $B_n$ is $PGL_n:=GL_n/Z(GL_n)$.
\end{proposition}
\begin{proof}
    The automorphism groups of the algebras $F_1(n)$ and $F_1$ are respectively $PGL_n$ and the trivial group $\{ e \}$. By direct computation, the linear isomorphisms of $B_n$ which are automorphisms of both $F_1(n)$ and $F_1$ (as Frobenius algebras) and that preserve the zipper and cozipper -- in other words, the automorphisms of $B_n$ -- consist in the group $\{ l \oplus 1 : l\in PGL_n\} \cong PGL_n$. 
\end{proof}

\begin{corollary}
    The category $\calcatname{C}_{\chi(1,n)}$ is equivalent to the category $\catname{Rep}(PGL_n)$.
\end{corollary}
\begin{proof}
    This follows from the proposition \ref{proposition-rep-automorphism-group}.
\end{proof}

This result is valid for all $n\in\mathds{N}$. However, the good category $\calcatname{C}_{\chi(1,d)}$ exists for any values of $d\in\mathds{K}$, we thus \textit{define} the interpolation family $\catname{Rep}(PGL_d):=\calcatname{C}_{\chi(1,d)}$.

\end{subsubsection}

\begin{subsubsection}{Case $t\neq1$} 
\label{section-rep-categories-case-tneq1}

Concerning $\calcatname{C}_{t\chi(1,d)}$, although it is not clear whether it relates to any famous categories, it can still be said that there is a KFA in the category $\catname{Rep}(PGL_d)\boxtimes\catname{Rep}(S_t)$ affording the character $t\chi(1,d)$. \\

Following table \ref{table:recap}, we can view $\catname{Rep}(PGL_d)$ and $\catname{Rep}(S_t)$ as respectively the categories $\calcatname{C}_{\chi_1}$ and $\calcatname{C}_{\chi_2}$, where 
\begin{align*}
    f_{\chi_1}(X,Y) = & \frac{1}{1-X}\frac{1}{1-dY}, &
    f_{\chi_2}(X,Y) = & \frac{t}{1-X}.
\end{align*}
By definition, the categories $\calcatname{C}_{\chi_1}$ and $\calcatname{C}_{\chi_2}$ are generated respectively by a KFA $(V_I,V_{S^1},\iota, \iota^*)$ and by a commutative Frobenius algebra $W$. The KFA structure in $\calcatname{C}_{\chi_1}\boxtimes\calcatname{C}_{\chi_2}$ affording $t\chi(1,d)$ is given by $(V_I\boxtimes W, V_{S^1}\boxtimes W, \iota\boxtimes \mathds{1}_W, \iota^*\boxtimes \mathds{1}_W)$, where the Frobenius algebra structures on $V_I\boxtimes W$ and $V_{S^1}\boxtimes W$ are given by the Deligne product of the respective structures, e.g., the unit of $V_I\boxtimes W$ is given by $u_I \boxtimes u_W$, where $u_I$ is the unit of $V_I$ and $u_W$ is the unit of $W$. \\

There is however no equivalence between $\calcatname{C}_{t\chi(1,d)}$ and $\catname{Rep}(PGL_d)\boxtimes\catname{Rep}(S_t)$ taking one KFA to the other because of mismatches in hom-space dimensions: we have for instance that the endomorphism spaces $\text{End}_{\calcatname{C}_{t\chi(1,d)}}(S^1)$ and $\text{End}_{\calcatname{C}_{\chi_1}\boxtimes\calcatname{C}_{\chi_2}}(V_{S^1}\boxtimes W)$ have dimensions $2$ and $4$ respectively.

\end{subsubsection}

\end{subsection}

\begin{subsection}{The category $\calcatname{C}_{\chi}^{\text{non-ss}}$}

Consider the character $\chi\in\text{Ch}_{\mathcal{T}_{\text{KFA}}}^{\mathcal{G}}$ with generating function 
\[
    f_{\chi}(X,Y) = \alpha_1 + \alpha_X X + \alpha_Y Y + \alpha_{Y^2} Y^2.
\]

The cases $\alpha_Y = 0 = \alpha_{Y^2}$ have been treated in \cite{kok} and are recapitulated in table \ref{table:recap}. We will use the following proposition.

\begin{proposition} (cf. lemma 2.6 in \cite{brundan-semisimplification})
    Let $\mathcal{C}$ be a semisimple good category, and $\overline{F}:(\mathcal{C}_{\text{univ}}^{\mathcal{T}_{\text{KFA}}})^{\text{IC}}\rightarrow \mathcal{C}$ be a full linear symmetric monoidal functor. Call $V$ the KFA in $\mathcal{C}$ defined by $\overline{F}$, and $\chi$ the character it affords. Then $\overline{F}$ factors through $\mathcal{C}_{\chi}$ via a fully faithful functor $F:\mathcal{C}_{\chi}\rightarrow \mathcal{C}$.
\end{proposition}
\begin{proof}
    We must show that the kernel of $\overline{F}$ is precisely the tensor ideal of negligible morphisms $\mathcal{N}_{\chi}$. Let $A$ and $B$ be objects of $\mathcal{C}_{\text{univ}}^{\mathcal{T}_{\text{KFA}}}$. For any map $ \overline{g}:\overline{F}B\rightarrow \overline{F}A$, $\exists g:B\rightarrow A$ such that $\overline{F}(g)=\overline{g}$. For any $f\in \mathcal{N}_{\chi}(A,B)$, we have by definition 
    $0=\chi(\text{Tr}(f\circ g))=\overline{F}(\text{Tr}(f\circ g)) = \text{Tr}(\overline{F}(f)\circ \overline{F}(g))=\text{Tr}(\overline{F}(f)\circ \overline{g})$. However, $\mathcal{C}$ is semisimple good, so the trace pairing is non-degenerate: this means that $\overline{F}(f)=0$, showing that $F$ is well-defined (and full). Conversely, if $f\in \mathcal{C}_{\text{univ}}^{\mathcal{T}_{\text{KFA}}}(A,B)$ is in the kernel of $\overline{F}$, then $\forall g:B\rightarrow A$, $\chi(\text{Tr}(f\circ g))=\overline{F}(\text{Tr}(f\circ g))=\text{Tr}(\overline{F}(f)\circ\overline{F}(g))=0$, so $f$ is $\chi$-negligible, and thus $f$ is $0$ in $\calcatname{C}_{\chi}$. This proves the faithfulness of $F$. 
\end{proof}

\begin{subsubsection}{Case $\alpha_{Y^2}\neq0$}
Consider the good category $\calcatname{R}:=\catname{Rep}(O_{\alpha_{Y^2}-2})\boxtimes \catname{Rep}(O_{\alpha_{X}-3})$ (see appendix \ref{section-annex-o}). Call $V$ and $W$ the generating objects of $\catname{Rep}(O_{\alpha_{Y^2}-2})$ and $\catname{Rep}(O_{\alpha_{X}-3})$ respectively; they are both endowed with a non-degenerate symmetric pairing, they form thus, according to example \ref{example-non-semisimple-general}, a KFA in $\calcatname{R}$. Following the notation from this example, set $\delta=\alpha_1$, $\sigma=\alpha_Y$, and take arbitrary $\alpha, \beta \in \mathds{K}^*$. Direct computations show that this KFA affords the character $\chi$.\\

According to proposition \ref{proposition-universal-property}, there is a linear monoidal functor $\overline{F}:(\mathcal{C}_{\text{univ}}^{\mathcal{T}_{\text{KFA}}})^{\text{IC}}\rightarrow \calcatname{R}$ sending the generating KFA $I\oplus S^1$ of $\mathcal{C}_{\text{univ}}^{\mathcal{T}_{\text{KFA}}}$ to the KFA constructed with $V$ and $W$. 

\begin{proposition}
    The functor $\overline{F}$ is full.
\end{proposition}
\begin{proof}
    Any morphism of $\calcatname{R}$ is made of linear combinations of composition and tensoring of the identities $\mathds{1}_V$, $\mathds{1}_W$, the copairings $\eta_V$, $\eta_W$ and the pairings $c_V$, $c_W$ (and of course morphisms coming from the monoidal structure). It is therefore sufficient to show that these morphisms are in the image of $\overline{F}$. More concretely, we must find that the idempotent endomorphisms $\mathds{P}_{V}$ of $A_{\alpha, \delta}(V)$ and $\mathds{P}_{W}$ of $A_{\beta,\sigma}(W')$ in $\calcatname{R}$, projecting respectively on $V$ and on $W$, can be written in terms of the KFA structure morphisms, i.e., they are in the image of $\overline{F}$. It will follow immediately that $\eta_V$, $\eta_W$, $c_V$, and $c_W$ are also in the image (being proportional to restrictions of the pairings and copairings of the Frobenius algebras $A_{\alpha, \delta}(V)$ and $A_{\beta,\sigma}(W')$). We have the following idempotents: 
    \begin{itemize}
        \item $\mathds{P}_{V}:=\mathds{1}_{A_{\alpha, \delta}(V)}-\mathds{P}_{U_I}-\mathds{P}_{N_I}$,
        \item $\mathds{P}_{W}:=\mathds{1}_{A_{\beta, \sigma}(W')}-\mathds{P}_{U_S}-\mathds{P}_{N_S}-\mathds{P}_{\textbf{1}}$,
    \end{itemize}
    where
    \begin{align*}
        \mathds{P}_{U_I} & := \frac{1}{\alpha_{Y^2}} u_{I} \circ \epsilon_{I} \circ H & \text{Im}(\mathds{P}_{U_I}) & \cong U_I, \\
        \mathds{P}_{N_I} & := \frac{1}{\alpha_{Y^2}} H \circ \left(u_{I} \circ \epsilon_{I} - \alpha_Y \right) & \text{Im}(\mathds{P}_{N_I}) & \cong N_I, \\ 
        \mathds{P}_{U_S} & := \frac{1}{\alpha_{Y^2}} u_{S^1} \circ \epsilon_{S^1} \circ W^2 & \text{Im}(\mathds{P}_{U_S}) & \cong U_S, \\ 
        \mathds{P}_{N_S} & := \frac{1}{\alpha_{Y^2}} W^2 \circ \left(u_{S^1} \circ \epsilon_{S^1} - \alpha_1 \right) & \text{Im}(\mathds{P}_{N_S}) & \cong N_S, \\
        \mathds{P}_{\textbf{1}} & := \frac{1}{\alpha_{Y^2}} W \circ \left( u_{S^1} \circ \epsilon_{S^1} \circ W - \alpha_Y \right) & \text{Im}(\mathds{P}_{\textbf{1}}) & \cong \textbf{1}   
    \end{align*}
    are orthogonal idempotents projecting on the summands of $A_{\alpha, \delta}(V):=U_I\oplus V \oplus N_I$ and $A_{\beta, \sigma}(W'):= U_S \oplus W \oplus \textbf{1} \oplus N_S$. Notice that they are all written in terms of the KFA struture morphisms, meaning that they are in the image of $\overline{F}$.
\end{proof}

\begin{corollary}
    The induced linear monoidal functor $F:\calcatname{C}_{\chi}\rightarrow \calcatname{R}$ is an equivalence of good categories.
\end{corollary}
\begin{proof}
    The above proof shows that both $V$ and $W$ are in the image of $\overline{F}$. Along with the morphisms $\eta_V$, $\eta_W$, $c_V$, and $c_W$, which are also in the image of $\overline{F}$, they generate the category $\calcatname{R}$. Therefore, $F$ is both fully faithful and essentially surjective on objects, i.e., it is an equivalence of categories. 
\end{proof}

The proof of fullness of $\overline{F}$ fails in the case $\alpha_{Y^2}=0$. In fact it is possible to show that the dimensions of the hom-spaces of $\calcatname{C}_{\chi}$ and $\calcatname{R}$ do not match in that case. 

\end{subsubsection}

\begin{subsubsection}{Case $\alpha_{Y^2}=0$, $\alpha_Y \neq 0$}

Consider the good category $\calcatname{S} := \calcatname{C}_{\chi_I} \boxtimes \catname{Rep}(O_{\alpha_X-2})$, where $\chi_I$ is the good character with generating function $f_{\chi_I}(X,Y)=\alpha_Y$. (The category $\calcatname{C}_{\chi_I} $ is equivalent to the category $\calcatname{C}$ described in \cite{kok}, cf. also table \ref{table:recap}.) \\

There is a KFA $(V_I,V_{S^1},\iota, \iota^*)$ in $\calcatname{S}$ affording the character $\chi$. Take $V_I$ to be the generating object of $\calcatname{C}_{\chi_I} $ and, for arbitrary $\alpha\in\mathds{K}^*$, $V_{S^1}=A_{\alpha, \alpha_1}(V)=U\oplus V\oplus N$, where $V$ is the generating object of $\catname{Rep}(O_{\alpha_X-2})$ and $U,N$ are isomorphic to the unit object $\textbf{1}$. Both $V_I $ and $ V_{S^1}$ have the structure of a commutative Frobenius algebra, which we distinguish naturally by adding subscripts $I$ and $S$ to the structure morphisms. We take the zipper and cozipper to be
\begin{itemize}
    \item \underline{Zipper} $\iota = u_I \boxtimes (U\cong \textbf{1})$,
    \item \underline{Cozipper} $\iota^* = \frac{1}{\alpha} \epsilon_I \boxtimes (\textbf{1}\cong N)$.
\end{itemize}

According to proposition \ref{proposition-universal-property}, there is a linear monoidal functor $\overline{F}:(\mathcal{C}_{\text{univ}}^{\mathcal{T}_{\text{KFA}}})^{\text{IC}}\rightarrow \calcatname{S}$ sending the generating KFA $I\oplus S^1$ of $\mathcal{C}_{\text{univ}}^{\mathcal{T}_{\text{KFA}}}$ to the KFA constructed above. 

\begin{proposition}
    The functor $\overline{F}$ is full.
\end{proposition}
\begin{proof}
    By a similar argument as for the case $\alpha_{Y^2}\neq 0$, it is sufficient to show that $V$ is in the image of $\overline{F}$. We have the following orthogonal idempotents in $\calcatname{S}$: 
    \begin{itemize}
        \item $\mathds{P}_U:=\frac{1}{\alpha_Y} u_{S} \circ \epsilon_I \circ \iota$,
        \item $\mathds{P}_N:=\frac{1}{\alpha_Y} ( \iota^* \circ u_I \circ \epsilon_{S} - \alpha_1 W) $,
        \item $\mathds{P}_V:=\mathds{1}_{V_{S^1}}-\mathds{P}_U-\mathds{P}_N$,
    \end{itemize}
    which are all in the image of $\overline{F}$ and project respectively on $U$, $N$, and $V$.
\end{proof}

\begin{corollary}
    The induced linear monoidal functor $F:\calcatname{C}_{\chi}\rightarrow \calcatname{S}$ is an equivalence of good categories. 
\end{corollary}

\end{subsubsection}

\end{subsection}

\end{section}

\newpage

\appendix
\begin{section}{The category $\catname{Rep}(GL_d)$ for $d\in\mathds{K}$}
\label{section-annex-gl}

Consider the type $t=((1,1))$ (only the identity box) and the empty theory $\mathcal{T}=\emptyset$ (no axioms). To lighten the notation, we draw the identity box by a simple straight line. 
We have 
\[ 
    \calcatname{C}^{\emptyset}_{\text{univ}}(\textbf{1},\textbf{1})=\mathds{K}\left[\input{rep-gl/identity-trace.tikz}\right]\equiv\mathds{K}\left[\input{rep-gl/identity-trace-circle.tikz}\right]. 
\]

It is shown in \cite{meir-interpolation} that for this theory, the algebra of good characters is isomorphic to $\mathds{K}$. The isomorphism is the following:
\[
    \mathds{K} \ni d \mapsto \chi_d  \qquad \text{where} \qquad \chi_d: \mathds{K}[X]_{\text{aug}}^{\emptyset}\rightarrow \mathds{K}, \quad \chi_d\left(\input{rep-gl/identity-trace-circle.tikz} \right) = d.
\]

The category $\mathcal{C}_{\chi_d}$ is identified in \cite{meir-interpolation} as Deligne's category $\catname{Rep}(GL_d)$.
In some sense, we have a description of $V$ for any "dimension" $d\in\mathds{K}$, namely the tautological object $V=V^{1,0}$ of $\catname{Rep}(GL_d)$. \\

\textit{Define} $\mathfrak{gl}_d:= V^{1,1}$ seen as an object of $\catname{Rep}(GL_d)$. 
Let $\alpha\in\mathds{K}^*$. We recall that for the unit object $\textbf{1}$ of $\catname{Rep}(GL_d)$, we have $\textbf{1}\otimes\textbf{1}=\textbf{1}$ and $\text{End}_{\catname{Rep}(GL_d)}(\textbf{1})=\mathds{K}$. The pair $V_I=\mathfrak{gl}_d$ and $V_{S^1}=\textbf{1}$, together with the following appropriate structure, forms a knowledgeable Frobenius algebra. 

\vspace{1em}
$\quad$ \underline{Open sector}:
\[
\begin{array}{rlrclll}
    \nabla_I: & \mathfrak{gl}_d\otimes \mathfrak{gl}_d\rightarrow \mathfrak{gl}_d, & \nabla_I & = & \input{rep-gl/GL-product.tikz} & \in & Con_{\emptyset}^{1+2,1+2} \\ 
    u_I: & \textbf{1} \rightarrow \mathfrak{gl}_d, & u_I & = & \hphantom{x}\begin{tikzpicture}
	\begin{pgfonlayer}{nodelayer}
		\node [style=none] (0) at (0, 0.75) {};
		\node [style=none] (1) at (0, -0.75) {};
		\node [style=none] (2) at (-0.25, 0) {};
		\node [style=none] (3) at (0.25, 0) {};
	\end{pgfonlayer}
	\begin{pgfonlayer}{edgelayer}
		\draw (0.center) to (1.center);
	\end{pgfonlayer}
\end{tikzpicture}
 & \in & Con_{\emptyset}^{1+0,1+0} \\ 
    \Delta_I: & \mathfrak{gl}_d\rightarrow \mathfrak{gl}_d\otimes \mathfrak{gl}_d, & \Delta_I & = & \alpha^{-1} \input{rep-gl/GL-product-braided.tikz} & \in & Con_{\emptyset}^{2+1,2+1} \\
    \epsilon_I: & \mathfrak{gl}_d\rightarrow \textbf{1}, & \epsilon_I & = & \alpha=\alpha~\text{tr} & \in & Con_{\emptyset}^{0+1,0+1} \\
\end{array}  
\]

$\quad$ \underline{Closed sector}:
\[
\begin{array}{rlrcccl}
    \nabla_{S^1}: & \textbf{1}\otimes \textbf{1}\rightarrow \textbf{1}, & \nabla_{S^1} & = & 1 & \in & \mathds{K} \\ 
    u_{S^1}: & \textbf{1} \rightarrow \textbf{1}, & u_{S^1} & = & 1 & \in & \mathds{K} \\ 
    \Delta_{S^1}: & \textbf{1}\rightarrow \textbf{1}\otimes \textbf{1}, & \Delta_{S^1} & = & \alpha^{-2} & \in & \mathds{K} \\
    \epsilon_{S^1}: & \textbf{1}\rightarrow \textbf{1}, & \epsilon_{S^1} & = & \alpha^2 & \in & \mathds{K} \\
\end{array}  
\]

$\quad$ \underline{Zipper}:
\[
\begin{array}{rllcccl}
    \iota: &  \textbf{1} \rightarrow \mathfrak{gl}_d, & \iota & = &  & \in & Con_{\emptyset}^{1+0,1+0} \\ 
    \iota^*: & \mathfrak{gl}_d \rightarrow \textbf{1}, & \iota^* & = & \alpha^{-1}~ & \in & Con_{\emptyset}^{0+1,0+1}
\end{array}  
\]

\end{section}
\begin{section}{The category $\catname{Rep}(O_d)$ for $d\in\mathds{K}$}
\label{section-annex-o}

Consider the algebraic type $t=((1,1),(2,0),(0,2))$ and theory $\mathcal{T}_{O}$: 
\[
    \begin{matrix}
        \input{figures/rep-O/O-U-box.tikz} \equiv \input{figures/rep-O/O-U.tikz} \qquad \input{figures/rep-O/O-n-box.tikz} \equiv \input{figures/rep-O/O-n.tikz} \qquad \input{figures/rep-O/O-id-box.tikz} \equiv \begin{tikzpicture}
	\begin{pgfonlayer}{nodelayer}
		\node [style=none] (0) at (0, 0.75) {};
		\node [style=none] (1) at (0, -0.75) {};
		\node [style=none] (2) at (-0.25, 0) {};
		\node [style=none] (3) at (0.25, 0) {};
	\end{pgfonlayer}
	\begin{pgfonlayer}{edgelayer}
		\draw (0.center) to (1.center);
	\end{pgfonlayer}
\end{tikzpicture}
 \\
        \input{figures/rep-O/O-n.tikz} = \input{figures/rep-O/O-n-symmetry.tikz} \qquad \input{figures/rep-O/O-axiom-non-degen-1.tikz} = \begin{tikzpicture}
	\begin{pgfonlayer}{nodelayer}
		\node [style=none] (0) at (0, 0.75) {};
		\node [style=none] (1) at (0, -0.75) {};
		\node [style=none] (2) at (-0.25, 0) {};
		\node [style=none] (3) at (0.25, 0) {};
	\end{pgfonlayer}
	\begin{pgfonlayer}{edgelayer}
		\draw (0.center) to (1.center);
	\end{pgfonlayer}
\end{tikzpicture}
  = \input{figures/rep-O/O-axiom-non-degen-2.tikz}
    \end{matrix}
\]

where we have replaced the boxes by straight lines to lighten the notation. \\

We have 
\[ 
    \calcatname{C}^{\mathcal{T}_{O}}_{\text{univ}}(\textbf{1},\textbf{1})=\mathds{K}\left[\input{rep-gl/identity-trace.tikz}\right]\equiv\mathds{K}\left[\input{rep-gl/identity-trace-circle.tikz}\right]. 
\]

It is shown in \cite{meir-interpolation} that for this theory, the algebra of good characters is isomorphic to $\mathds{K}$. The isomorphism is the following:
\[
    \mathds{K} \ni d \mapsto \chi_d  \qquad \text{where} \qquad \chi_d: \mathds{K}[X]_{\text{aug}}^{\mathcal{T}_{O}}\rightarrow \mathds{K}, \quad \chi_d\left(\input{rep-gl/identity-trace-circle.tikz} \right) = d.
\]

The category $\mathcal{C}_{\chi_d}$ is identified in \cite{meir-interpolation} as Deligne's category $\catname{Rep}(O_d)$.
In some sense, we have a description of $V$ for any "dimension" $d\in\mathds{K}$, namely the tautological object $V=V^{1,0}$ of $\catname{Rep}(O_d)$. \\

In the category $\catname{Rep}(O_d)$, the object $V$ has a non-degenerate symmetric pairing, $c$. One can also create artificially another object, $V':=V\oplus \Gamma$, where $\Gamma=\mathbf{1}$, endowed with a non-degenerate symmetric pairing $c'=c \oplus 1$, where $1$ is the identity $\Gamma \otimes \Gamma = \textbf{1}\otimes \textbf{1} = \textbf{1}=\Gamma$.

\end{section}

\newpage
\begin{center}\textbf{References}\end{center}

\vspace{-7em}

\bibliographystyle{alpha} 
\bibliography{biblio}

\begin{thebibliography}{BEAEO20}

\bibitem[BEAEO20]{brundan-semisimplification}
Jonathan Brundan, Inna Entova-Aizenbud, Pavel Etingof, and Victor Ostrik.
\newblock Semisimplification of the category of tilting modules for {GL}.
\newblock {\em Advances in Mathematics}, 375:107331, dec 2020.

\bibitem[CMS23]{meir-invariant-r-spin}
Nils Carqueville, Ehud Meir, and Lorant Szegedy.
\newblock Invariants of r-spin tqfts and non-semisimplicity, 2023.

\bibitem[Del07]{deligne-category-gl}
Pierre Deligne.
\newblock {La Categorie des Representations du Groupe Symetrique {$S_t$},
  lorsque {$t$} n'est pas un Entier Naturel}.
\newblock {\em Algebraic groups and homogeneous spaces}, 01 2007.

\bibitem[Koc03]{kock-book}
Joachim Kock.
\newblock {\em Frobenius Algebras and 2-D Topological Quantum Field Theories}.
\newblock London Mathematical Society Student Texts. Cambridge University
  Press, 2003.

\bibitem[KOK22]{kok}
Mikhail Khovanov, Victor Ostrik, and Yakov Kononov.
\newblock Two-dimensional topological theories, rational functions and their
  tensor envelopes.
\newblock {\em Selecta Mathematica}, 28(4):71, 2022.

\bibitem[LP08]{lauda-pfeiffer-1}
Aaron~D. Lauda and Hendryk Pfeiffer.
\newblock {Open--closed strings: Two-dimensional extended TQFTs and Frobenius
  algebras}.
\newblock {\em Topology and its Applications}, 155(7):623--666, 2008.

\bibitem[Mei21]{meir-universal}
Ehud Meir.
\newblock {Universal Rings of Invariants}.
\newblock {\em International Mathematics Research Notices}, 05 2021.

\bibitem[Mei23]{meir-interpolation}
Ehud Meir.
\newblock Interpolations of monoidal categories and algebraic structures by
  invariant theory.
\newblock {\em Selecta Mathematica}, 29(4):58, 2023.

\bibitem[SS22]{stern-szegedy-2022}
Walker~H. Stern and L{\'o}r{\'a}nt Szegedy.
\newblock Topological field theories on open-closed r-spin surfaces.
\newblock {\em Topology and its Applications}, 312:108062, 2022.

\bibitem[Vog99]{vogel-universal}
Pierre Vogel.
\newblock {The Universal Lie Algebra}.
\newblock \url{https://webusers.imj-prg.fr/~pierre.vogel/grenoble-99b.pdf},
  1999.

\end{thebibliography}


\end{document}